\documentclass[twocolumn]{autart}    

\usepackage{amsmath} 
\usepackage{amssymb}  

\usepackage[T1]{fontenc}

\newenvironment{proof}{\emph{Proof.}}{\hfill $\square$ \\}

\newtheorem{algo}[thm]{Algorithm}

\newtheorem{example}[thm]{Example}

\renewcommand\labelenumi{(\roman{enumi})}
\renewcommand\theenumi\labelenumi

\begin{document}

\begin{frontmatter}

\title{A Dual Geometric Test for Forward-Flatness} 

\author{Bernd Kolar\corauthref{cor}}\ead{bernd.kolar@jku.at},    
\author{Johannes Schrotshamer}\ead{johannes.schrotshamer@jku.at},  
\author{Markus Sch{\"o}berl}\ead{markus.schoeberl@jku.at}              

\corauth[cor]{Corresponding author.}

\address{Institute of Control Systems, Johannes Kepler University, Linz, Austria}


\begin{abstract}                          
Forward-flatness is a generalization of static feedback linearizability and a special case of a more general flatness concept for discrete-time systems. Recently, it has been shown that this practically quite relevant property can be checked by computing a unique sequence of involutive distributions which generalizes the well-known static feedback linearization test.
In this paper, a dual test for forward-flatness based on a unique sequence of integrable codistributions is derived. Since the main mathematical operations for determining this sequence are the intersection of codistributions and the calculation of Lie derivatives of 1-forms, it is computationally quite efficient.
Furthermore, the formulation with codistributions also facilitates a comparison with the existing discrete-time literature regarding the closely related topic of dynamic feedback linearization, which is mostly formulated in terms of 1-forms rather than vector fields.
The presented results are illustrated by two examples.
\end{abstract}

\end{frontmatter}

\section{Introduction}

The concept of flatness has been introduced by Fliess, L{\'e}vine,
Martin and Rouchon in the 1990s for nonlinear continuous-time systems,
see e.g. \cite{FliessLevineMartinRouchon:1992}, \cite{FliessLevineMartinRouchon:1995},
or \cite{FliessLevineMartinRouchon:1999}. Flat continuous-time systems
possess the characteristic feature that all system variables can be
expressed by a flat output and its time derivatives, which allows
elegant solutions for trajectory planning and trajectory tracking
problems. Because of its high practical relevance with a wide variety
of applications, flatness has attracted a lot of attention in the
control systems community. Nevertheless, checking the flatness of
nonlinear multi-input systems is a highly nontrivial problem, for
which still no complete systematic solution in the form of verifiable
necessary and sufficient conditions exists (see e.g. \cite{NicolauRespondek:2016},
\cite{NicolauRespondek:2017}, or \cite{GstottnerKolarSchoberl:2021b}).

In view of the inherent discrete-time nature of digital control circuits,
applying flatness-based methods to discrete-time systems is not only
interesting from a theoretical but also from a practical perspective.
Furthermore, a discrete-time approach is also favorable for a combination
with data-based methods, see e.g. \cite{AlsaltiBerberichLopezAllgoewerMueller:2021}
or \cite{EckerSchoeberl:2023}. However, it should be noted that regarding
the flatness of nonlinear discrete-time systems there exist two approaches.
The first one is to replace the time derivatives of the continuous-time
definition by forward-shifts as e.g. in \cite{KaldmaeKotta:2013},
\cite{Sira-RamirezAgrawal:2004}, or \cite{KolarKaldmaeSchoberlKottaSchlacher:2016}.
This point of view is also consistent with the notion of endogenous
and exogenous dynamic feedback as it is defined in \cite{Aranda-BricaireMoog:2008}.
The second approach considers flatness as the existence of a one-to-one
correspondence between the system trajectories and the trajectories
of a trivial system. It is a generalization of the first approach,
since here the flat output may also depend on backward-shifts of the
system variables, see \cite{DiwoldKolarSchoberl:2020} or \cite{GuillotMillerioux:2020}.
To ensure a clear distinction, we refer to the first approach, which
we consider in the present paper, as forward-flatness.

In \cite{KolarSchoberlDiwold:2019} it has been shown that forward-flat
systems can be decomposed by state- and input transformations into
a subsystem and an endogenous dynamic feedback, such that the complete
system is forward-flat if and only if the subsystem is forward-flat.
As proposed in \cite{KolarSchoberlDiwold:2019}, a repeated application
of this property allows to check if a system is forward-flat by decomposing
it step by step into subsystems of decreasing size. If finally only
a trivial (i.e., empty) system is left, the original system is forward-flat,
and a flat output can be obtained. Otherwise, if in some step the
considered subsystem does not allow a further decomposition, it can
be concluded that this subsystem as well as the original system are
not forward-flat. However, for deriving the transformations which
achieve these decompositions, it is necessary to straighten out certain
distributions by the Frobenius theorem, which is computationally tedious
and requires the solution of nonlinear ordinary or linear partial
differential equations. Indeed, it would be computationally more efficient
to only check if the repeated decompositions are possible, without
actually performing them. Based on this idea, it has been shown in
\cite{KolarDiwoldSchoberl:2019} that forward-flatness can be checked
by computing a unique sequence of involutive distributions, which
generalizes the sequence of distributions from the static feedback
linearization test of \cite{Grizzle:1986}. The existence of such
a systematic test is surprising since there does not exist a counterpart
in the continuous-time case. The purpose of the present paper is now
to derive a dual test for forward-flatness, which is based on a unique
sequence of integrable codistributions. Our motivation is twofold:
First, deriving a dual approach is of interest in its own right, like
the dual version of the well-known static feedback linearization test
for continuous-time systems (e.g. \cite{GardnerShadwick:1992}, \cite{Sluis:1993b},
\cite{TilburySastry:1994}, \cite{Sastry:1999}). Second, the paper
is motivated by the fact that the majority of the discrete-time literature
regarding the closely related topics of dynamic feedback linearization
and controllability uses an algebraic framework based on 1-forms,
see e.g. \cite{Grizzle:1993}, \cite{Aranda-BricaireKottaMoog:1996},
or \cite{Aranda-BricaireMoog:2008}. Even though we use differential-geometric
methods, the codistributions of the proposed sequence are also spanned
by 1-forms. Thus, the dual approach should be accessible to a wider
audience and also facilitate a comparison with the existing literature.
A rather recent paper which uses an algebraic framework based on 1-forms
is \cite{Kaldmae:2021}, where necessary and sufficient conditions
for flatness of discrete-time systems in the more general sense including
backward-shifts are proposed. However, as mentioned there, even when
applied to the special case of forward-flat systems the computational
complexity is higher since partial differential equations have to
be solved, whereas in the present contribution only the intersection
of codistributions and computing Lie derivatives of 1-forms will be
needed. Another important aspect of our approach is that for a given
discrete-time system the proposed sequence of codistributions is uniquely
defined, i.e., like in \cite{KolarDiwoldSchoberl:2019} there occur
no degrees of freedom during the computations.

The paper is organized as follows: After an overview of some basic
differential-geometric concepts and the notation in Section \ref{sec:Notation},
in Section \ref{sec:Discrete-time-Systems} we recapitulate the concept
of forward-flatness for discrete-time systems. Subsequently, Section
\ref{sec:Invariant-Codistributions} addresses invariant codistributions
and Cauchy characteristics, which form the mathematical foundation
upon which the paper is based. In Section \ref{sec:Dual-Test-Forward-Flatness}
we then introduce the sequence of codistributions, and show how it
is related to system decompositions and forward-flatness. Finally,
Section \ref{sec:Example} illustrates our results by two examples.

\section{\protect\label{sec:Notation}Notation and Geometric Preliminaries}

Throughout the paper we make use of basic differential-geometric concepts.
The following section shall provide a brief overview and introduce
the employed notation. For an introduction to differential-geometric
concepts with a focus on nonlinear control applications we refer e.g.
to \cite{NijmeijervanderSchaft:1990} or \cite{Sastry:1999}, and
for a rather general introduction to differential geometry e.g. to
\cite{Boothby:1986} or \cite{Lee:2012}.

Let $\mathcal{M}$ denote an $n$-dimensional manifold with local
coordinates $x^{1},\ldots,x^{n}$. A vector field $v$ on $\mathcal{M}$
has the form $v=v^{i}(x)\partial_{x^{i}}$, where $\partial_{x^{i}}$,
$i=1,\ldots,n$ denotes the basis vector fields corresponding to the
local coordinates, and $v^{i}(x)$, $i=1,\ldots,n$ are smooth functions.
To keep formulas short and readable we make use of the Einstein summation
convention, i.e., the summation symbol is omitted when the index range
is clear from the context. The dual object of a vector field is a
1-form $\omega=\omega_{i}(x)\mathrm{d}x^{i}$, where $\mathrm{d}x^{i}$,
$i=1,\ldots,n$ are the differentials corresponding to the local coordinates,
and $\omega_{i}(x)$, $i=1,\ldots,n$ are again smooth functions.
Given a vector field $v$ and a 1-form $\omega$, their contraction
(interior product) is defined by $v\rfloor\omega=v^{i}(x)\omega_{i}(x)$.
The Lie derivative of a smooth function $f(x)$ along a vector field
$v$ is given by $L_{v}f=v^{i}(x)\partial_{x^{i}}f(x)$, and the Lie
derivative of a 1-form $\omega$ is given by $L_{v}\omega=(L_{v}\omega_{i})\mathrm{d}x^{i}+\omega_{i}(\mathrm{d}L_{v}x^{i})=(v^{k}\partial_{x^{k}}\omega_{i})\mathrm{d}x^{i}+\omega_{i}\mathrm{d}v^{i}$,
with $\mathrm{d}$ denoting the exterior derivative. Like for a smooth
function, the Lie derivative of a 1-form $\omega$ describes the change
of the 1-form when following the flow generated by the vector field
$v$. The Lie derivative of a vector field $w$ along a vector field
$v$ is given by the Lie bracket $[v,w]$.

The Lie derivative can also be applied to higher-order differential
forms. For a $k$-form $\alpha$, the Lie derivative along a vector
field $v$ is denoted by $L_{v}\alpha$, and the exterior derivative
of a $k$-form $\alpha$ yields a $(k+1)$-form $\mathrm{d}\alpha$.
In particular, the exterior derivative of a 0-form (function) $f$
gives the 1-form $\mathrm{d}f$. The contraction between a vector
field and a $k$-form $\alpha$ is denoted like in the case of a 1-form
by $v\rfloor\alpha$, and yields a $(k-1)$-form. For a 1-form $\omega$
and a vector field $v$, an important connection between Lie derivative,
contraction, and exterior derivative is given by Cartan's magic formula
$L_{v}\omega=v\rfloor\mathrm{d}\omega+\mathrm{d}(v\rfloor\omega)$.
A further important concept is the wedge product (exterior product)
of differential forms. If $\alpha$ is a $k$-form and $\beta$ a
$p$-form, then $\alpha\wedge\beta$ is a $(k+p)$-form. In particular,
if $\omega^{1},\ldots,\omega^{k}$ are 1-forms, then $\omega^{1}\wedge\ldots\wedge\omega^{k}$
is a $k$-form. This $k$-form is nonzero if and only if the 1-forms
$\omega^{1},\ldots,\omega^{k}$ are linearly independent.

If a set of 1-forms $\omega^{1},\ldots,\omega^{p}$ with some $p$
is linearly independent, then $P=\mathrm{span}\{\omega^{1},\ldots,\omega^{p}\}$
is a $p$-dimensional codistribution on the $n$-dimensional manifold
$\mathcal{M}$, where span denotes the span over the ring of smooth
functions $C^{\infty}(\mathcal{M})$. Likewise, if $v_{1},\ldots,v_{d}$
with some $d$ is a set of linearly independent vector fields, then
$D=\mathrm{span}\{v_{1},\ldots,v_{d}\}$ is a $d$-dimensional distribution.
The annihilator $P^{\perp}$ of a $p$-dimensional codistribution
$P$ is the unique $(n-p)$-dimensional distribution consisting of
all vector fields $v$ that meet $v\rfloor\omega=0$ for all 1-forms
$\omega\in P$. Conversely, the annihilator $D^{\perp}$ of a $d$-dimensional
distribution $D$ is the unique $(n-d$)-dimensional codistribution
consisting of all 1-forms $\omega$ that meet $v\rfloor\omega=0$
for all vector fields $v\in D$.

A $p$-dimensional codistribution $P=\mathrm{span}\{\omega^{1},\ldots,\omega^{p}\}$
is called integrable if $\mathrm{d}\omega^{i}\wedge\omega^{1}\wedge\ldots\wedge\omega^{p}=0$
for all $i=1,\ldots,p$. In this case, according to the Frobenius
theorem (version for codistributions, see e.g. \cite{Sastry:1999}),
there exist functions $f^{1}(x),\ldots,f^{p}(x)$ such that locally
$P=\mathrm{span}\{\mathrm{d}f^{1},\ldots,\mathrm{d}f^{p}\}$. Moreover,
since the differentials $\mathrm{d}f^{1},\ldots,\mathrm{d}f^{p}$
are linearly independent, it is possible to introduce local coordinates
$\bar{x}=\Phi(x)$ such that $f^{i}=\bar{x}^{i}$, $i=1,\ldots,p$.
In such coordinates $P=\mathrm{span}\{\mathrm{d}\bar{x}^{1},\ldots,\mathrm{d}\bar{x}^{p}\}$,
and we say that $P$ is straightened out. Similarly, a $d$-dimensional
distribution $D=\mathrm{span}\{v_{1},\ldots,v_{d}\}$ is called involutive
if $[v_{i},v_{j}]\in D$ for all $i,j\in\{1,\ldots,d\}$. In this
case, the Frobenius theorem (version for distributions, see e.g. \cite{NijmeijervanderSchaft:1990})
guarantees the existence of a coordinate transformation $\bar{x}=\Phi(x)$
such that locally $D=\mathrm{span}\{\partial_{\bar{x}^{1}},\ldots,\partial_{\bar{x}^{d}}\}$.
Again, we say that in these coordinates $D$ is straightened out.
This can be achieved in two steps. First, an involutive distribution
can always be written in the form $D=\mathrm{span}\{\bar{v}_{1},\ldots,\bar{v}_{d}\}$
with a basis consisting of vector fields that meet $[\bar{v}_{i},\bar{v}_{j}]=0$
for all $i,j\in\{1,\ldots,d\}$. Second, for any single vector field
$v=v^{i}(x)\partial_{x^{i}}$, the flow-box theorem allows to construct
a coordinate transformation $\bar{x}=\Phi(x)$ such that locally $v=\partial_{\bar{x}^{1}}$,
i.e., $v$ is straightened out (see e.g. \cite{NijmeijervanderSchaft:1990}).
Since all pairwise Lie brackets of the vector fields $\bar{v}_{1},\ldots,\bar{v}_{d}$
vanish, it is possible to straighten out these vector fields simultaneously,
i.e., $\bar{v}_{i}=\partial_{\bar{x}^{i}}$ for $i=1,\ldots,d$, and
the representation $D=\mathrm{span}\{\partial_{\bar{x}^{1}},\ldots,\partial_{\bar{x}^{d}}\}$
follows.

Regarding notation, it should also be noted that the symbols $\subset$
and $\supset$ are used in the sense that they include equality. Finally,
it is important to emphasize that throughout the paper we assume that
the dimensions of the considered codistributions and distributions
as well as the ranks of the considered Jacobian matrices are locally
constant.

\section{\protect\label{sec:Discrete-time-Systems}Discrete-time Systems and
Forward-flatness}

In this contribution, we consider nonlinear time-invariant discrete-time
systems
\begin{equation}
x^{i,+}=f^{i}(x,u)\,,\quad i=1,\ldots,n\label{eq:sys}
\end{equation}
with $\dim(x)=n$, $\dim(u)=m$, and smooth functions $f(x,u)$ that
satisfy the submersivity condition
\begin{equation}
\mathrm{rank}(\partial_{(x,u)}f)=n\,.\label{eq:submersivity}
\end{equation}
The assumption (\ref{eq:submersivity}) is quite common in the discrete-time
literature and necessary for accessibility (see e.g. \cite{Grizzle:1993}).
The superscript $+$ on the left-hand side of the system equations
(\ref{eq:sys}) denotes the forward-shift of the corresponding variables.
For indicating also higher-order forward-shifts, we use subscripts
in brackets. For instance, the $\alpha$-th forward-shift of a component
$u^{j}$ of the input with some $\alpha\in\mathbb{N}$ is denoted
by $u_{[\alpha]}^{j}$, and $u_{[\alpha]}=(u_{[\alpha]}^{1},\ldots,u_{[\alpha]}^{m})$.
In order to summarize the concept of forward-flatness, we make use
of a space with coordinates $(x,u,u_{[1]},u_{[2]},\ldots)$. If $g$
is a smooth function defined on this space, then its future values
can be determined by a repeated application of the forward-shift operator,
which is defined according to the rule
\[
\delta(g(x,u,u_{[1]},u_{[2]},\ldots))=g(f(x,u),u_{[1]},u_{[2]},u_{[3]},\ldots)\,.
\]
To define a backward-shift operator $\delta^{-1}$, in general it
would be necessary to extend the system map (\ref{eq:sys}) such that
it becomes invertible. However, in this contribution we only need
backward-shifts of functions of the form $h(f(x,u))$, and in this
case the backward-shift is obviously given by
\[
\delta^{-1}(h(f(x,u))=h(x)\,.
\]
The shift operators can also be applied to 1-forms by shifting both
their coefficients and differentials. In particular, we will need
the backward-shift of 1-forms of the form $\omega_{i}(f(x,u))\mathrm{d}f^{i}$,
which is given by
\begin{equation}
\delta^{-1}(\omega_{i}(f(x,u))\mathrm{d}f^{i})=\omega_{i}(x)\mathrm{d}x^{i}\,.\label{eq:backward_shift_1form}
\end{equation}
For codistributions which are spanned by 1-forms of the form $\omega_{i}(f(x,u))\mathrm{d}f^{i}$,
the backward-shift is defined accordingly. It is also important to
emphasize that all our results are local since we apply the inverse-
and the implicit function theorem as well as the Frobenius theorem,
which allow only local results. Thus, like the discrete-time static
feedback linearization problem, we consider discrete-time flatness
in a suitable neighborhood of an equilibrium $(x_{0},u_{0})$ of the
system (\ref{eq:sys}). Since the map (\ref{eq:sys}) is continuous,
it can then be guaranteed that the system trajectories do not leave
the regions of validity of the coordinate transformations derived
by the above-mentioned theorems (see e.g. \cite{Grizzle:1986}, \cite{NijmeijervanderSchaft:1990},
or \cite{Kotta:1995}). Forward-flatness can now be defined as follows.
\begin{defn}
\label{def:forward-flatness}The system (\ref{eq:sys}) is said to
be forward-flat around an equilibrium $(x_{0},u_{0})$, if the $n+m$
coordinate functions $x$ and $u$ can be expressed locally by an
$m$-tuple of functions
\begin{equation}
y^{j}=\varphi^{j}(x,u,u_{[1]},\ldots,u_{[q]})\,,\quad j=1,\ldots,m\label{eq:flat_output}
\end{equation}
and their forward-shifts $y_{[1]}=\delta(\varphi(x,u,u_{[1]},\ldots,u_{[q]}))$,
$y_{[2]}=\delta^{2}(\varphi(x,u,u_{[1]},\ldots,u_{[q]}))$, $\ldots$
up to some finite order. The $m$-tuple (\ref{eq:flat_output}) is
called a flat output.
\end{defn}

It should be noted that the number of components
of a flat output is equal to the number of components of the input
of the system, i.e., $\dim(y)=\dim(u)=m$. The representation of
$x$ and $u$ by the flat output and its forward-shifts is unique,
and has the form\footnote{The multi-index $R=(r_{1},\ldots,r_{m})$ contains the number of forward-shifts
of the individual components of the flat output which is needed to
express $x$ and $u$, and $y_{[0,R]}$ is an abbreviation for $y$
and its forward-shifts up to order $R$.}
\begin{equation}
\begin{array}{cclcl}
x^{i} & = & F_{x}^{i}(y_{[0,R-1]})\,, & \quad & i=1,\ldots,n\\
u^{j} & = & F_{u}^{j}(y_{[0,R]})\,, & \quad & j=1,\ldots,m\,.
\end{array}\label{eq:flat_param}
\end{equation}
The term forward-flatness refers to the fact that both in the flat
output (\ref{eq:flat_output}) as well as in the corresponding parameterization
of the system variables (\ref{eq:flat_param}) there occur forward-shifts
but no backward-shifts like in the more general case discussed in
\cite{DiwoldKolarSchoberl:2020} or \cite{GuillotMillerioux:2020}.
Like for differentially flat continuous-time systems, it can be shown
that the map $(x,u)=F(y_{[R]})$ given by (\ref{eq:flat_param}) is
a submersion, i.e., that its Jacobian matrix has linearly independent
rows. If the system (\ref{eq:sys}) is static feedback linearizable
and $y=\varphi(x)$ a linearizing output, the submersion (\ref{eq:flat_param})
becomes a diffeomorphism. The proof is analogous to the continuous-time
case, which can be found e.g. in \cite{Kolar:2017}.
\begin{example}
\label{exa:runningexample_flatness}The system
\begin{equation}
\begin{array}{ccl}
x^{1,+} & = & u^{1}-x^{2}\\
x^{2,+} & = & x^{1}(u^{1}-u^{2})\\
x^{3,+} & = & u^{2}
\end{array}\label{eq:runningexample_sys}
\end{equation}
with $\dim(x)=3$ and $\dim(u)=2$ is forward-flat around the equilibrium
$x_{0}=(\tfrac{1}{2},\tfrac{1}{2},0)$, $u_{0}=(1,0)$. A flat output
is given by
\begin{equation}
y=(x^{1}-x^{3},x^{2})\,,\label{eq:runningexample_flat_output}
\end{equation}
and the corresponding map (\ref{eq:flat_param}) reads
\begin{equation}
\begin{array}{ccl}
x^{1} & = & \tfrac{y_{[1]}^{2}}{y_{[1]}^{1}+y^{2}}\\
x^{2} & = & y^{2}\\
x^{3} & = & -\tfrac{y^{1}y_{[1]}^{1}+y^{1}y^{2}-y_{[1]}^{2}}{y_{[1]}^{1}+y^{2}}\\
u^{1} & = & \tfrac{y_{[2]}^{1}y^{2}+y^{2}y_{[1]}^{2}+y_{[2]}^{2}}{y_{[2]}^{1}+y_{[1]}^{2}}\\
u^{2} & = & -\tfrac{y_{[1]}^{1}y_{[2]}^{1}+y_{[1]}^{1}y_{[1]}^{2}-y_{[2]}^{2}}{y_{[2]}^{1}+y_{[1]}^{2}}\,.
\end{array}\label{eq:runningexample_flat_param}
\end{equation}
This can be verified by substituting the flat output (\ref{eq:runningexample_flat_output})
and its forward-shifts into (\ref{eq:runningexample_flat_param}).
\end{example}

As proven in \cite{KolarSchoberlDiwold:2019}, forward-flat systems
can always be transformed into a certain triangular form.
\begin{thm}
\label{thm:basic_decomposition_flat}A forward-flat system (\ref{eq:sys})
can be transformed by a state- and input transformation\begin{subequations}\label{eq:decomposition_coord_transformation}
\begin{align}
(\bar{x}_{1},\bar{x}_{2})=\:\: & \Phi_{x}(x)\label{eq:decompostion_state_transformation}\\
(\bar{u}_{1},\bar{u}_{2})=\:\: & \Phi_{u}(x,u)\label{eq:decompostion_input_transformation}
\end{align}
\end{subequations}into the form\begin{subequations}
\label{eq:basic_decomposition_flat}
\begin{align}
\bar{x}_{2}^{+} & =f_{2}(\bar{x}_{2},\bar{x}_{1},\bar{u}_{2})\label{eq:decomposition_subsystem}\\
\bar{x}_{1}^{+} & =f_{1}(\bar{x}_{2},\bar{x}_{1},\bar{u}_{2},\bar{u}_{1})\label{eq:decomposition_feedback}
\end{align}
\end{subequations} with $\dim(\bar{x}_{1})\geq1$ and $\mathrm{rank}(\partial_{\bar{u}_{1}}f_{1})=\dim(\bar{x}_{1})$.
\end{thm}

\begin{proof}
For systems (\ref{eq:sys}) with $\mathrm{rank}(\partial_{u}f)=m$,
i.e., systems without redundant inputs, this is shown in Theorem 6
of \cite{KolarSchoberlDiwold:2019}. In the case $\mathrm{rank}(\partial_{u}f)<m$,
redundant inputs can always be eliminated by an input transformation
without affecting the forward-flatness of the system, see e.g. Lemma
9 of \cite{KolarSchoberlDiwold:2019}. Thus, combining Theorem 6 and
Lemma 9 of \cite{KolarSchoberlDiwold:2019} completes the proof for
general systems (\ref{eq:sys}) with $\mathrm{rank}(\partial_{u}f)\leq m$.
\end{proof}
\begin{example}
\label{exa:runningexample_triangular_decomposition}Consider the system
(\ref{eq:runningexample_sys}) of Example \ref{exa:runningexample_flatness}.
Applying the state- and input transformation\footnote{How to derive the transformation (\ref{eq:decomposition_coord_transformation})
systematically will be shown in the proof of Theorem \ref{thm:relation_sequence_decomposition}.
For an alternative approach based on involutive distributions see
\cite{KolarSchoberlDiwold:2019}.}
\[
\begin{array}{cclccl}
\bar{x}_{2}^{1} & = & x^{1}-x^{3}\quad\quad & \bar{u}_{2}^{1} & = & u^{1}-u^{2}\\
\bar{x}_{2}^{2} & = & x^{2} & \bar{u}_{1}^{1} & = & u^{2}\\
\bar{x}_{1}^{1} & = & x^{3}
\end{array}
\]
transforms the system into the form
\[
\begin{array}{ccl}
\bar{x}_{2}^{1,+} & = & \bar{u}_{2}^{1}-\bar{x}_{2}^{2}\\
\bar{x}_{2}^{2,+} & = & (\bar{x}_{2}^{1}+\bar{x}_{1}^{1})\bar{u}_{2}^{1}\\
\bar{x}_{1}^{1,+} & = & \bar{u}_{1}^{1}\,,
\end{array}
\]
which corresponds to (\ref{eq:basic_decomposition_flat}) with $\bar{x}_{1}=\bar{x}_{1}^{1}$,
$\bar{x}_{2}=(\bar{x}_{2}^{1},\bar{x}_{2}^{2})$, $\bar{u}_{1}=\bar{u}_{1}^{1}$,
and $\bar{u}_{2}=\bar{u}_{2}^{1}$. Hence, $\dim(\bar{x}_{1})=1$,
$\dim(\bar{x}_{2})=2$, $\dim(\bar{u}_{1})=1$, and $\dim(\bar{u}_{2})=1$.\footnote{Because of $\dim(\bar{x}_{1})=1$, $\dim(\bar{u}_{1})=1$,
and $\dim(\bar{u}_{2})=1$, the upper indices in $\bar{x}_{1}^{1}$,
$\bar{u}_{1}^{1}$, and $\bar{u}_{2}^{1}$ could also be omitted.
However, for consistency reasons, we write the upper index even if
the corresponding block of variables is only 1-dimensional.}
\end{example}

The importance of the triangular form (\ref{eq:basic_decomposition_flat})
is due to the fact that the equations (\ref{eq:decomposition_subsystem})
can be considered as a subsystem with inputs $(\bar{x}_{1},\bar{u}_{2})$,
and that this subsystem is forward-flat if and only if the complete
system (\ref{eq:basic_decomposition_flat}) is forward-flat.
\begin{lem}
\label{lem:flatness_properties_basic_decomposition}A system of the
form (\ref{eq:basic_decomposition_flat}) with $\mathrm{rank}(\partial_{\bar{u}_{1}}f_{1})=\dim(\bar{x}_{1})$
is forward-flat if and only if the subsystem (\ref{eq:decomposition_subsystem})
with the inputs $(\bar{x}_{1},\bar{u}_{2})$ is forward-flat. The
flat outputs are related as follows:
\begin{enumerate}
\item If $\dim(\bar{u}_{1})=\dim(\bar{x}_{1})$, then every flat output
$y_{2}$ of (\ref{eq:decomposition_subsystem}) is also a flat output
$y$ of (\ref{eq:basic_decomposition_flat}).
\item \label{enu:case_redundant_inputs_basic_decomposition}If $\dim(\bar{u}_{1})>\dim(\bar{x}_{1})$,
then the components of $\bar{u}_{1}$ can be split in the form $\bar{u}_{1}=(\hat{u}_{1},\tilde{u}_{1})$
such that $\mathrm{rank}(\partial_{\hat{u}_{1}}f_{1})=\dim(\bar{x}_{1})$,
and a flat output of (\ref{eq:basic_decomposition_flat}) is given
by $y=(y_{2},\tilde{u}_{1})$.
\end{enumerate}
\end{lem}

\begin{proof}
Because of $\mathrm{rank}(\partial_{\bar{u}_{1}}f_{1})=\dim(\bar{x}_{1})$,
it is always possible to choose the input transformation (\ref{eq:decompostion_input_transformation})
such that with $\bar{u}_{1}=(\hat{u}_{1},\tilde{u}_{1})$ the equations
(\ref{eq:decomposition_feedback}) have the form $\bar{x}_{1}^{+}=\hat{u}_{1}$.
Since (\ref{eq:decomposition_feedback}) is then only a simple prolongation
of the input variables $\bar{x}_{1}$ of the subsystem (\ref{eq:decomposition_subsystem}),
the claims follow immediately from Definition \ref{def:forward-flatness}.
\end{proof}
In the case $\dim(\bar{u}_{1})=\dim(\bar{x}_{1})$, the equations
(\ref{eq:decomposition_feedback}) can be considered as an endogenous
dynamic feedback for the subsystem (\ref{eq:decomposition_subsystem}).
The term ``endogenous'' reflects the fact that every trajectory
$(\bar{x}_{2}(k),\bar{x}_{1}(k),\bar{u}_{2}(k))$ of the subsystem
(\ref{eq:decomposition_subsystem}) uniquely determines a trajectory
of the complete system (\ref{eq:basic_decomposition_flat}). Because
of $\dim(\bar{u}_{1})=\dim(\bar{x}_{1})$ and $\mathrm{rank}(\partial_{\bar{u}_{1}}f_{1})=\dim(\bar{x}_{1})$,
the corresponding sequence $\bar{u}_{1}(k)$ can be calculated immediately
from the equations (\ref{eq:decomposition_feedback}). In the case
$\dim(\bar{u}_{1})>\dim(\bar{x}_{1})$, the trajectory of the input
variables $\tilde{u}_{1}$ can still be chosen arbitrarily, reflecting
the fact that a flat output of (\ref{eq:basic_decomposition_flat})
is given by $y=(y_{2},\tilde{u}_{1})$.
\begin{rem}
\label{rem:static_feedback_decomposition}A special case of forward-flat
systems are systems which are linearizable by a static feedback. Such
systems allow (repeated) decompositions of the form (\ref{eq:basic_decomposition_flat})
with $\bar{u}_{2}$ empty, see \cite{NijmeijervanderSchaft:1990}.
\end{rem}

\section{\protect\label{sec:Invariant-Codistributions}Invariant Codistributions
and Cauchy Characteristics}

The test for forward-flatness which we propose in Section~\ref{sec:Dual-Test-Forward-Flatness}
is based on the concept of invariant codistributions. The notion of
invariant distributions and codistributions is used quite frequently
in control theory, see e.g. \cite{NijmeijervanderSchaft:1990}. A
typical field of application is the analysis of the controllability
and observability properties of nonlinear continuous-time systems.
In this section, we recapitulate basic facts and formulate some technical
results which we need in the main part of the paper. As already mentioned
in Section \ref{sec:Notation}, throughout the paper we assume that
all considered distributions and codistributions have locally constant
dimension.

A $p$-dimensional codistribution $P=\mathrm{span}\{\omega^{1},\ldots,\omega^{p}\}$,
defined on some $n$-dimensional manifold $\mathcal{M}$ with local
coordinates $(x^{1},\ldots,x^{n})$, is called invariant w.r.t. a
vector field $v$ if $L_{v}\omega\in P$ for all 1-forms $\omega\in P$.
This condition is often abbreviated as $L_{v}P\subset P$, cf. \cite{NijmeijervanderSchaft:1990}.
Accordingly, $P$ is called invariant w.r.t. a $d$-dimensional distribution
$D=\mathrm{span}\{v_{1},\ldots,v_{d}\}$ if $L_{v}\omega\in P$ for
all 1-forms $\omega\in P$ and vector fields $v\in D$. It is straightforward
to verify that this is the case if and only if
\[
L_{v_{i}}\omega^{j}\in P\,,\quad i=1,\ldots,d\,,\,j=1,\ldots,p
\]
for arbitrary bases $\{v_{1},\ldots,v_{d}\}$ of $D$ and $\{\omega^{1},\ldots,\omega^{p}\}$
of $P$. Furthermore, it is an immediate consequence of the definition
of invariance that also all higher order (repeated and mixed) Lie
derivatives of 1-forms $\omega\in P$ w.r.t. vector fields $v\in D$
are contained in $P$.
\begin{example}
\label{exa:invariance}Let $P=\mathrm{span}\{\omega^{1},\omega^{2}\}$
with $\omega^{1}=x^{3}\mathrm{d}x^{2}$ and $\omega^{2}=-x^{2}x^{4}\mathrm{d}x^{1}+(x^{4})^{2}\mathrm{d}x^{4}$
be a codistribution on a manifold with coordinates $(x^{1},x^{2},x^{3},x^{4})$
and $v=\partial_{x^{3}}$ a vector field. Since the Lie derivatives
$L_{v}\omega^{1}=\mathrm{d}x^{2}$ and $L_{v}\omega^{2}=0$ are both
contained in $P$, the codistribution is invariant w.r.t. $v$.
\end{example}

Given a vector field $v$ and a codistribution $P$ which is not invariant
w.r.t. $v$, one can pose the question how to extend $P$ such that
it becomes invariant. Indeed, there always exists a unique smallest
invariant codistribution which contains $P$. Computing the smallest
invariant codistribution which contains a given codistribution is
used in the control literature e.g. for checking the observability
of nonlinear continuous-time systems, see \cite{NijmeijervanderSchaft:1990}.
In the following, we consider the general case of invariance w.r.t.
a distribution, and show that the invariant extension can be constructed
by adding suitable (also higher-order) Lie derivatives of the 1-forms
of an arbitrary basis of $P$.
\begin{prop}
\label{prop:invariant_extension}Given a codistribution $P$ and a
distribution $D$, there exists a unique codistribution $\hat{P}$
of minimal dimension which contains $P$ and is invariant w.r.t. $D$.
\end{prop}

\begin{proof}
According to the definition of invariance, every invariant extension
of $P$ must contain the Lie derivatives $L_{v_{i}}\omega^{j}$ of
the 1-forms $\omega^{1},\ldots,\omega^{p}$ of $P$ w.r.t. all vector
fields $v_{1},\ldots,v_{d}$ of $D$. Hence, defining
\begin{equation}
P^{(1)}\negthinspace\negthinspace=\negthinspace\mathrm{span}\{\omega^{1},\ldots,\omega^{p},L_{v_{i}}\omega^{j},i\negthinspace=\negthinspace1,\ldots,d,j\negthinspace=\negthinspace1,\ldots,p\},\label{eq:invariant_extension_first_step}
\end{equation}
the smallest invariant extension $\hat{P}$ of $P$ meets $P^{(1)}\subset\hat{P}$.
Now let $\{\bar{\omega}^{1},\ldots,\bar{\omega}^{\bar{p}}\}$ with
$\bar{p}=\dim(P^{(1)})$ be a basis for $P^{(1)}$, i.e., a maximal
set of linearly independent 1-forms selected from the 1-forms in (\ref{eq:invariant_extension_first_step}).
Then we can define
\[
P^{(2)}\negthinspace\negthinspace=\negthinspace\mathrm{span}\{\bar{\omega}^{1},\ldots,\bar{\omega}^{\bar{p}},L_{v_{i}}\bar{\omega}^{j},i\negthinspace=\negthinspace1,\ldots,d,j\negthinspace=\negthinspace1,\ldots,\bar{p}\},
\]
and with the same argumentation as before, $P^{(2)}\subset\hat{P}$
follows. This procedure can be repeated until for some $k$ we have
$P^{(k+1)}=P^{(k)}$ , which means that $P^{(k)}$ is invariant w.r.t.
the distribution $D$. Because of $P^{(k)}\subset\hat{P}$ and the
invariance of $P^{(k)}$, we then actually have $P^{(k)}=\hat{P}$,
i.e., $P^{(k)}$ is the smallest invariant extension $\hat{P}$ of
$P$. Since the dimension of a codistribution cannot exceed the dimension
of the underlying manifold, $k\leq n-p$ with $n=\dim(\mathcal{M})=n$
and $p=\dim(P)$.
\end{proof}
From the proof of Proposition \ref{prop:invariant_extension}, it
can be observed that the smallest invariant extension $\hat{P}$ always
possesses a basis which consists only of the 1-forms $\omega^{1},\ldots,\omega^{p}$
of the original codistribution $P$ and their (higher order, repeated
and mixed) Lie derivatives. The following example illustrates the
construction of the smallest invariant extension of a codistribution.
\begin{example}
\label{exa:invariance_continued}Consider the codistribution $P$
of Example \ref{exa:invariance}, and the distribution $D=\mathrm{span}\{v_{1},v_{2}\}$
with $v_{1}=\partial_{x^{3}}$ and $v_{2}=\partial_{x^{4}}$. Since
all of the Lie derivatives
\begin{equation}
\begin{array}{ccl}
L_{v_{1}}\omega^{1} & = & \mathrm{d}x^{2}\\
L_{v_{1}}\omega^{2} & = & 0\\
L_{v_{2}}\omega^{1} & = & 0\\
L_{v_{2}}\omega^{2} & = & -x^{2}\mathrm{d}x^{1}+2x^{4}\mathrm{d}x^{4}
\end{array}\label{eq:example_invariance_continued_Lie_derivatives-1}
\end{equation}
except $L_{v_{2}}\omega^{2}$ are contained in $P$, the basis of
$P$ has to be extended by the 1-form $L_{v_{2}}\omega^{2}$. Now
it has to be checked if the resulting codistribution $P^{(1)}=\mathrm{span}\{\omega^{1},\omega^{2},L_{v_{2}}\omega^{2}\}$
is already invariant w.r.t. the distribution $D$. By construction,
all Lie derivatives (\ref{eq:example_invariance_continued_Lie_derivatives-1})
are contained in $P^{(1)}$. Since furthermore the Lie derivatives
\begin{align*}
L_{v_{1}}(L_{v_{2}}\omega^{2}) & =0\\
L_{v_{2}}^{2}\omega^{2} & =2\mathrm{d}x^{4}
\end{align*}
of the added 1-form $L_{v_{2}}\omega^{2}$ are also contained in $P^{(1)}$,
the smallest codistribution which contains $P$ and is invariant w.r.t.
$D$ is given by $\hat{P}=P^{(1)}$.
\end{example}

If the distribution $D$ is involutive and meets $D\rfloor P=0$,
which means that $v\rfloor\omega=0$ for all $v\in D$ and $\omega\in P$,
then the following result for $\hat{P}$ can be shown.
\begin{prop}
Consider a codistribution $P$ and a distribution $D$, and let $\hat{P}$
denote the smallest codistribution which contains $P$ and is invariant
w.r.t. $D$. If $D\rfloor P=0$ and $D$ is involutive, then also
$D\rfloor\hat{P}=0$.
\end{prop}

\begin{proof}
The proof is based on the fact that a basis for $\hat{P}$ can be
constructed by simply adding suitable Lie derivatives to a basis of
$P$ as shown above. According to Cartan's magic formula, the Lie
derivative of a 1-form $\omega$ along a vector field $v$ is given
by
\begin{equation}
L_{v}\omega=v\rfloor\mathrm{d}\omega+\mathrm{d}(v\rfloor\omega)\,.\label{eq:Cartan_magic_formula}
\end{equation}
Since $v\rfloor\omega=0$ holds for all 1-forms $\omega$ of $P$
and vector fields $v$ of $D$, the first-order Lie derivatives are
of the form $L_{v}\omega=v\rfloor\mathrm{d}\omega$. With $w$ denoting
another arbitrary vector field of $D$ and using the identity\footnote{See e.g. \cite{Lee:2012}.}
\[
w\rfloor\left(v\rfloor\mathrm{d}\omega\right)=L_{v}\left(w\rfloor\omega\right)-L_{w}\left(v\rfloor\omega\right)-\left[v,w\right]\rfloor\omega\,,
\]
we get
\[
w\rfloor L_{v}\omega=L_{v}\left(w\rfloor\omega\right)-L_{w}\left(v\rfloor\omega\right)-\left[v,w\right]\rfloor\omega\,.
\]
Because of $D\rfloor P=0$ and the involutivity of $D$, which implies
$\left[v,w\right]\in D$, all terms on the right-hand side vanish
and hence $w\rfloor L_{v}\omega=0$. Consequently, since $w$ is an
arbitrary vector field of $D$, the first-order Lie derivatives meet
$D\rfloor L_{v}\omega=0$. Repeating this argumentation shows that
this is also true for all higher-order Lie derivatives, and hence
$D\rfloor\hat{P}=0$.
\end{proof}
With the additional property $D\rfloor P=0$, invariance is closely
related to the notion of Cauchy-characteristic vector fields and distributions,
which is discussed in detail e.g. in \cite{Schoberl:2014}, \cite{Stormark:2000},
or \cite{BryantChernGardnerGoldschmidtGriffiths:1991}.
\begin{defn}
(\cite{Schoberl:2014}) A vector field $v$ is called a Cauchy-characteristic
vector field of a codistribution $P$ if\footnote{Here $v\rfloor\mathrm{d}P\subset P$ is an abbreviation for $v\rfloor\mathrm{d}\omega\subset P\,,\,\forall\omega\in P$.}
\begin{equation}
v\rfloor P=0\quad\text{and}\quad v\rfloor\mathrm{d}P\subset P\,.\label{eq:condition_cauchy}
\end{equation}
The set of all Cauchy-characteristic vector fields forms the Cauchy-characteristic
distribution $\mathcal{C}(P)$, which is involutive.
\end{defn}

The importance of Cauchy-characteristic vector fields lies in the
existence of coordinate transformations such that $P$ can be represented
by a reduced number of coordinates. More precisely, if an arbitrary
vector field $v\in\mathcal{C}(P)$ is straightened out by the flow-box
theorem using a coordinate transformation $\bar{x}=\Phi(x)$ such
that $v=\partial_{\bar{x}^{1}}$, then there exists a basis for $P$
which is independent of $\bar{x}^{1}$. Since $\mathcal{C}(P)$ is
involutive, the Frobenius theorem allows to straighten out even the
whole distribution. In such coordinates, there exists a basis for
$P$ which is independent of the corresponding $\dim(\mathcal{C}(P))$
coordinates (see \cite{Schoberl:2014}, \cite{Stormark:2000}, or
\cite{BryantChernGardnerGoldschmidtGriffiths:1991}).
\begin{example}
\label{exa:Cauchy}Consider the codistribution
\[
P=\mathrm{span}\{\omega^{1},\omega^{2}\}
\]
with $\omega^{1}=\mathrm{d}x^{2}+x^{1}\mathrm{d}x^{3}$ and $\omega^{2}=\mathrm{d}x^{1}-\mathrm{d}x^{3}$
as well as the vector field $v=\partial_{x^{1}}-x^{1}\partial_{x^{2}}+\partial_{x^{3}}$
on a manifold with coordinates $(x^{1},x^{2},x^{3})$. Since $v\rfloor\omega^{1}=0$,
$v\rfloor\omega^{2}=0$, and the 1-forms
\[
\begin{array}{ccl}
v\rfloor\mathrm{d}\omega^{1} & = & -\mathrm{d}x^{1}+\mathrm{d}x^{3}\\
v\rfloor\mathrm{d}\omega^{2} & = & 0
\end{array}
\]
are contained in $P$, the vector field $v$ meets the condition (\ref{eq:condition_cauchy})
and is hence a Cauchy-characteristic vector field of $P$. An application
of the flow-box theorem yields a coordinate transformation $(\bar{x}^{1},\bar{x}^{2},\bar{x}^{3})=(x^{1},\tfrac{(x^{1})^{2}}{2}+x^{2},x^{3}-x^{1})$
such that $v=\partial_{\bar{x}^{1}}$. In these coordinates, the 1-forms
$\omega^{1}$ and $\omega^{2}$ which span $P$ are given by $\omega^{1}=\mathrm{d}\bar{x}^{2}+\bar{x}^{1}\mathrm{d}\bar{x}^{3}$
and $\omega^{2}=-\mathrm{d}\bar{x}^{3}$. By constructing suitable
linear combinations, the codistribution can indeed be written in the
form
\[
P=\mathrm{span}\{\mathrm{d}\bar{x}^{2},\mathrm{d}\bar{x}^{3}\}\,,
\]
with a basis that is independent of $\bar{x}^{1}$.
\end{example}

In the main part of the paper, we make use of the following straightforward
observation.
\begin{prop}
\label{prop:D_subset_Cauchy}If a codistribution $P$ is invariant
w.r.t. a distribution $D$ which meets $D\rfloor P=0$, then $D$
is a subdistribution of the Cauchy-characteristic distribution $\mathcal{C}(P)$
of $P$.
\end{prop}

\begin{proof}
The Lie derivative of a 1-form $\omega$ along a vector field $v$
is given by (\ref{eq:Cartan_magic_formula}). Thus, every vector field
$v$ with $v\rfloor P=0$ which meets the invariance condition $L_{v}P\subset P$
also meets the condition (\ref{eq:condition_cauchy}) for a Cauchy-characteristic
vector field of $P$. Consequently, $D\subset\mathcal{C}(P)$.
\end{proof}
In contrast to the Cauchy-characteristic distribution, the distribution
$D$ of Proposition \ref{prop:D_subset_Cauchy} is not necessarily
involutive. However, if $D$ is involutive, it can also be straightened
out by the Frobenius theorem, and then because of $D\subset\mathcal{C}(P)$
there exists a basis for $P$ which is independent of the corresponding
$\dim(D)$ coordinates. The following corollary summarizes and combines
the content of this section such that we can directly apply it in
the remainder of the paper.
\begin{cor}
\label{cor:Invariance_Cauchy_summarized}Consider a codistribution
$P$ and an involutive distribution $D$ with $D\rfloor P=0$. Then
the following holds.
\begin{enumerate}
\item \label{enu:Invariance_Cauchy_summarized_1}There exists a unique smallest
codistribution $\hat{P}$ which contains $P$ and is invariant w.r.t.
$D$.
\item \label{enu:Invariance_Cauchy_summarized_2}The codistribution $\hat{P}$
meets $D\rfloor\hat{P}=0$.
\item \label{enu:Invariance_Cauchy_summarized_3}The distribution $D$ is
contained in the Cauchy-characteristic distribution of $\hat{P}$,
i.e., $D\subset\mathcal{C}(\hat{P})$.
\item \label{enu:Invariance_Cauchy_summarized_4}After performing a coordinate
transformation $\bar{x}=\Phi(x)$ such that $D=\mathrm{span}\{\partial_{\bar{x}^{1}},\ldots,\partial_{\bar{x}^{d}}\}$,
there exists a basis for $\hat{P}$ which is independent of $\bar{x}^{1},\ldots,\bar{x}^{d}$.
\end{enumerate}
\end{cor}

\section{\protect\label{sec:Dual-Test-Forward-Flatness}Dual Test for Forward-flatness}

In this section we introduce a certain sequence of codistributions,
and subsequently show how it is related to system decompositions and
forward-flatness.

\subsection{Definition of the Sequence of Codistributions}

In the following, we define a sequence of codistributions which allows
to efficiently check the forward-flatness of nonlinear discrete-time
systems (\ref{eq:sys}). For a given system the sequence is uniquely
determined, i.e., there occur no degrees of freedom. The main mathematical
operations are the intersection of codistributions, the calculation
of Lie derivatives of 1-forms to determine the smallest invariant
extensions of codistributions according to Proposition \ref{prop:invariant_extension},
and backward-shifts of 1-forms. Hence, solving ODEs or even PDEs is
not required. The calculations are performed on the $(n+m)$-dimensional
state- and input manifold $\mathcal{X}\times\mathcal{U}$ with coordinates
$(x,u)$. In order to simplify the calculations as much as possible,
it can be convenient to introduce adapted coordinates
\begin{equation}
\begin{array}{ccl}
\theta^{i} & = & f^{i}(x,u)\,,\quad i=1,\ldots,n\\
\xi^{j} & = & h^{j}(x,u)\,,\quad j=1,\ldots,m
\end{array}\label{eq:adapted_coordinates}
\end{equation}
instead of $x$ and $u$, since then $\mathrm{span}\{\mathrm{d}f\}=\mathrm{span}\{\mathrm{d}\theta\}$
and $\mathrm{span}\{\mathrm{d}f\}^{\perp}=\mathrm{span}\{\partial_{\xi}\}$.
Because of the submersivity property (\ref{eq:submersivity}), there
always exist functions $h^{j}(x,u)$ such that the Jacobian matrix
of the right-hand side of (\ref{eq:adapted_coordinates}) is regular
and the transformation hence invertible.

\begin{algo}\label{alg:definition_sequence}

Start with $P_{1}=\mathrm{span}\{\mathrm{d}x\}$, and repeat the following
steps for $k\geq1$:\textbf{\vspace{-0.2cm}
}
\begin{enumerate}
\item[\emph{1.}] Compute the intersection $P_{k}\cap\mathrm{span}\{\mathrm{d}f\}$.
\item[\emph{2.}] Determine the smallest codistribution $P_{k+1}^{+}$ which is invariant
w.r.t. the distribution $\mathrm{span}\{\mathrm{d}f\}^{\perp}$ and
contains $P_{k}\cap\mathrm{span}\{\mathrm{d}f\}$.\footnote{In case $P_{k}\cap\mathrm{span}\{\mathrm{d}f\}$ is already invariant
w.r.t. $\mathrm{span}\{\mathrm{d}f\}^{\perp}$, this step is trivial
since $P_{k+1}^{+}=P_{k}\cap\mathrm{span}\{\mathrm{d}f\}$.}
\item[\emph{3.}] Define $P_{k+1}=\delta^{-1}(P_{k+1}^{+})$.\textbf{\vspace{-0.1cm}
\vspace{-0.1cm}
}
\end{enumerate}
Stop if $P_{\bar{k}+1}=P_{\bar{k}}$ for some $k=\bar{k}$.

\end{algo}

The algorithm defines a unique sequence of codistributions: Given
$P_{k}$, the intersection of Step 1 is clearly unique. According
to Proposition \ref{prop:invariant_extension} with $D=\mathrm{span}\{\mathrm{d}f\}^{\perp}$
and $P=P_{k}\cap\mathrm{span}\{\mathrm{d}f\}$, the smallest invariant
extension $P_{k+1}^{+}$ in Step 2 (corresponding to $\hat{P}$ of
Proposition \ref{prop:invariant_extension}) is also unique. Finally,
also the result of the backward-shift in Step 3 is unique. Since in
(\ref{eq:backward_shift_1form}) we have introduced the backward-shift
only for 1-forms of the form $\omega_{i}(f(x,u))\mathrm{d}f^{i}$,
it is important to show that $P_{k+1}^{+}$ indeed possesses a basis
consisting of such 1-forms. Because of item \ref{enu:Invariance_Cauchy_summarized_2}
of Corollary \ref{cor:Invariance_Cauchy_summarized}, like \emph{$P_{k}\cap\mathrm{span}\{\mathrm{d}f\}$}
also the codistribution $P_{k+1}^{+}$ is contained in $\mathrm{span}\{\mathrm{d}f\}$.
Moreover, because of item \ref{enu:Invariance_Cauchy_summarized_3}
and \ref{enu:Invariance_Cauchy_summarized_4} of Corollary \ref{cor:Invariance_Cauchy_summarized},
$P_{k+1}^{+}$ possesses even a basis with 1-forms of the form $\omega_{i}(f(x,u))\mathrm{d}f^{i}$,
where also the coefficients $\omega_{i}$ only depend on the functions
$f(x,u)$. This is obvious in adapted coordinates (\ref{eq:adapted_coordinates})
with $\mathrm{span}\{\mathrm{d}f\}=\mathrm{span}\{\mathrm{d}\theta\}$
and $\mathrm{span}\{\mathrm{d}f\}^{\perp}=\mathrm{span}\{\partial_{\xi}\}$.
Thus, in Step 3, a basis for $P_{k+1}$ is obtained by applying (\ref{eq:backward_shift_1form})
to these 1-forms.
\begin{example}
Let us compute the codistribution $P_{2}$ for the system (\ref{eq:runningexample_sys})
of Example \ref{exa:runningexample_flatness}.\\
Step 1: The intersection of $P_{1}=\mathrm{span}\{\mathrm{d}x^{1},\mathrm{d}x^{2},\mathrm{d}x^{3}\}$
and $\mathrm{span}\{\mathrm{d}f\}=\mathrm{span}\{-\mathrm{d}x^{2}+\mathrm{d}u^{1},(u^{1}-u^{2})\mathrm{d}x^{1}+x^{1}\mathrm{d}u^{1}-x^{1}\mathrm{d}u^{2},\mathrm{d}u^{2}\}$
is given by
\[
P_{1}\cap\mathrm{span}\{\mathrm{d}f\}=\mathrm{span}\{\omega^{1}\}
\]
with
\[
\begin{array}{ccc}
\omega^{1} & = & (u^{1}-u^{2})\mathrm{d}x^{1}+x^{1}\mathrm{d}x^{2}\end{array}\,.
\]
Step 2: The annihilator of $\mathrm{span}\{\mathrm{d}f\}$ is given
by
\[
\mathrm{span}\{\mathrm{d}f\}^{\perp}=\mathrm{span}\{v_{1},v_{2}\}
\]
with
\[
\begin{array}{ccl}
v_{1} & = & \partial_{x^{3}}\\
v_{2} & = & -x^{1}\partial_{x^{1}}+(u^{1}-u^{2})\partial_{x^{2}}+(u^{1}-u^{2})\partial_{u^{1}}\,.
\end{array}
\]
In order to determine the smallest codistribution which is invariant
w.r.t. $\mathrm{span}\{\mathrm{d}f\}^{\perp}$ and contains $P_{1}\cap\mathrm{span}\{\mathrm{d}f\}$,
we have to compute the Lie derivatives
\begin{equation}
\begin{array}{ccl}
L_{v_{1}}\omega^{1} & = & 0\\
L_{v_{2}}\omega^{1} & = & (u^{1}-u^{2})\mathrm{d}x^{1}+x^{1}\mathrm{d}u^{1}-x^{1}\mathrm{d}u^{2}
\end{array}\label{eq:runningexample_algorithm_Lie_derivatives_1}
\end{equation}
of $\omega^{1}$ w.r.t. the vector fields $v_{1}$ and $v_{2}$. Since
$L_{v_{2}}\omega^{1}$ is not contained in $P_{1}\cap\mathrm{span}\{\mathrm{d}f\}$,
the codistribution is not invariant and we have to extend its basis
by the 1-form $L_{v_{2}}\omega^{1}$. However, it can be verified
that the resulting codistribution $\mathrm{span}\{\omega^{1},L_{v_{2}}\omega^{1}\}$
is already invariant w.r.t. $\mathrm{span}\{\mathrm{d}f\}^{\perp}$.
The Lie derivatives (\ref{eq:runningexample_algorithm_Lie_derivatives_1})
are contained in $\mathrm{span}\{\omega^{1},L_{v_{2}}\omega^{1}\}$
by construction, and the Lie derivatives
\[
\begin{array}{rcc}
L_{v_{1}}(L_{v_{2}}\omega^{1}) & = & 0\\
L_{v_{2}}^{2}\omega^{1} & = & 0
\end{array}
\]
of the added 1-form $L_{v_{2}}\omega^{1}$ are obviously also contained
in $\mathrm{span}\{\omega^{1},L_{v_{2}}\omega^{1}\}$. Thus, we have
\[
P_{2}^{+}=\mathrm{span}\{\omega^{1},L_{v_{2}}\omega^{1}\}\,.
\]
Step 3: As discussed above, $P_{2}^{+}$ possesses a basis with 1-forms
of the form $\omega_{i}(f(x,u))\mathrm{d}f^{i}$. Indeed, it can be
written in the form
\[
P_{2}^{+}=\mathrm{span}\{-\mathrm{d}f_{1}+\mathrm{d}f_{3},\mathrm{d}f_{2}\}\,.
\]
Applying the backward-shift operator (\ref{eq:backward_shift_1form})
finally yields
\[
P_{2}=\delta^{-1}(P_{2}^{+})=\mathrm{span}\{-\mathrm{d}x^{1}+\mathrm{d}x^{3},\mathrm{d}x^{2}\}\,.
\]
It should be noted that here all calculations were performed in the
original coordinates $(x,u)$. However, as mentioned above, it can
be convenient to use adapted coordinates (\ref{eq:adapted_coordinates}).
We will demonstrate the application of Algorithm \ref{alg:definition_sequence}
with adapted coordinates by an example in Section \ref{sec:Example}.
\end{example}

\begin{rem}
Note that the computational effort is lower than in \cite{KolarDiwoldSchoberl:2019},
where the calculation of the corresponding sequence of distributions
requires the calculation of largest projectable subdistributions in
every step. The computation of the smallest invariant codistributions
in Step 2 of Algorithm \ref{alg:definition_sequence} by just adding
Lie derivatives of 1-forms can be considered here as the simpler task.
\end{rem}

In order to show that the stop condition of Algorithm \ref{alg:definition_sequence}
is reasonable, we prove now the following.
\begin{prop}
The codistributions $P_{1},\ldots,P_{\bar{k}}$ form a nested sequence
\begin{equation}
P_{\bar{k}}\subset P_{\bar{k}-1}\subset\ldots\subset P_{1}\,.\label{eq:sequence}
\end{equation}
\end{prop}

\begin{proof}
First, let us show that all codistributions $P_{1},\ldots,P_{\bar{k}}$
are contained in $\mathrm{span}\{\mathrm{d}x\}$. Indeed, an application
of Corollary \ref{cor:Invariance_Cauchy_summarized} with $D=\mathrm{span}\{\mathrm{d}f\}^{\perp}$
and $P=P_{k}\cap\mathrm{span}\{\mathrm{d}f\}$ shows that \emph{$P_{k+1}^{+}\subset\mathrm{span}\{\mathrm{d}f\}$}
for all $k\geq1$. Because of $\delta^{-1}(\mathrm{d}f)=\mathrm{d}x$,
\[
P_{k+1}\subset\mathrm{span}\{\mathrm{d}x\}\,,\quad k\geq1
\]
follows. Next, we show that the codistributions $P_{1},\ldots,P_{\bar{k}}$
form a nested sequence. For $k=1$, because of $P_{1}=\mathrm{span}\{\mathrm{d}x\}$
this directly implies $P_{2}\subset P_{1}$. Now assume that for some
$k>1$ we have $P_{k}\subset P_{k-1}$. Then obviously also $P_{k}\cap\mathrm{span}\{\mathrm{d}f\}\subset P_{k-1}\cap\mathrm{span}\{\mathrm{d}f\}$
holds, and the invariant extensions computed in Step 2 of the procedure
meet $P_{k+1}^{+}\subset P_{k}^{+}$ (an invariant codistribution
which contains $P_{k-1}\cap\mathrm{span}\{\mathrm{d}f\}$ must also
contain $P_{k}\cap\mathrm{span}\{\mathrm{d}f\}\subset P_{k-1}\cap\mathrm{span}\{\mathrm{d}f\}$).
Applying the backward-shift operator to this relation yields $P_{k+1}\subset P_{k}$,
and by induction we finally get (\ref{eq:sequence}).
\end{proof}
Next, we address the integrability of the codistributions~(\ref{eq:sequence}).
\begin{prop}
The codistributions of the sequence (\ref{eq:sequence}) are integrable.
\end{prop}

\begin{proof}
For $k=1$, the codistribution $P_{1}=\mathrm{span}\{\mathrm{d}x\}$
is clearly integrable. In the following, we prove that for $k\geq1$
the integrability of $P_{k}$ implies the integrability of $P_{k+1}$.
In fact, we only need to prove the integrability of \emph{$P_{k+1}^{+}$},
since the application of the backward-shift operator in Step 3 of
the procedure does not affect the integrability. For this purpose,
let
\[
P_{k}\cap\mathrm{span}\{\mathrm{d}f\}=\mathrm{span}\{\omega^{1},\ldots,\omega^{d_{1}}\}
\]
and
\[
P_{k}=\mathrm{span}\{\omega^{1},\ldots,\omega^{d_{1}},\mu^{1},\ldots,\mu^{d_{2}}\}\,,
\]
with 1-forms $\mu$ that are not contained in $\mathrm{span}\{\mathrm{d}f\}$.
Furthermore, let
\begin{equation}
P_{k+1}^{+}=\mathrm{span}\{\omega^{1},\ldots,\omega^{d_{1}},\rho^{1},\ldots,\rho^{d_{3}}\}\,,\label{eq:Pk+1+}
\end{equation}
with the 1-forms $\rho$ denoting suitable Lie derivatives of the
1-forms $\omega$ according to the construction of an invariant codistribution
discussed in the proof of Proposition \ref{prop:invariant_extension}.
First, it is important to note that
\begin{equation}
\omega^{1}\wedge\ldots\wedge\omega^{d_{1}}\wedge\mu^{1}\wedge\ldots\wedge\mu^{d_{2}}\wedge\rho^{1}\wedge\ldots\wedge\rho^{d_{3}}\neq0\label{eq:linear_independence_wedge}
\end{equation}
since all these 1-forms are linearly independent. The 1-forms $\omega^{1},\ldots,\omega^{d_{1}},\rho^{1},\ldots\rho^{d_{3}}$
are clearly linearly independent since they form a basis for $P_{k+1}^{+}\subset\mathrm{span}\{\mathrm{d}f\}$,
and after adding $\mu^{1},\ldots,\mu^{d_{2}}$ the linear independence
still holds since there exists no linear combination of the latter
1-forms which is contained in $\mathrm{span}\{\mathrm{d}f\}$. The
codistribution (\ref{eq:Pk+1+}) is now integrable if and only if
\begin{equation}
\mathrm{d}\omega^{s}\wedge\omega^{1}\wedge\ldots\wedge\omega^{d_{1}}\wedge\rho^{1}\wedge\ldots\wedge\rho^{d_{3}}=0\,,\quad s=1,\ldots,d_{1}\label{eq:WedgeProduct_1}
\end{equation}
as well as
\begin{equation}
\mathrm{d}\rho^{s}\wedge\omega^{1}\wedge\ldots\wedge\omega^{d_{1}}\wedge\rho^{1}\wedge\ldots\wedge\rho^{d_{3}}=0\,,\quad s=1,\ldots,d_{3}\,.\label{eq:WedgeProduct_2}
\end{equation}

To prove (\ref{eq:WedgeProduct_1}), we make use of the assumption
that $P_{k}$ is integrable, which implies
\[
\mathrm{d}\omega^{s}\wedge\omega^{1}\wedge\ldots\wedge\omega^{d_{1}}\wedge\mu^{1}\wedge\ldots\wedge\mu^{d_{2}}=0\,,
\]
and because of (\ref{eq:linear_independence_wedge}) also
\begin{equation}
\mathrm{d}\omega^{s}\wedge\omega^{1}\wedge\ldots\wedge\omega^{d_{1}}\wedge\mu^{1}\wedge\ldots\wedge\mu^{d_{2}}\wedge\rho^{1}\wedge\ldots\wedge\rho^{d_{3}}=0\,.\label{eq:dw_wedge_all}
\end{equation}
In the following, it is convenient to introduce adapted coordinates
(\ref{eq:adapted_coordinates}) since then $\mathrm{span}\{\mathrm{d}f\}=\mathrm{span}\{\mathrm{d}\theta\}$
and $\mathrm{span}\{\mathrm{d}f\}^{\perp}=\mathrm{span}\{\partial_{\xi}\}$.
In such coordinates, the 1-forms $\omega$ have the form
\[
\omega^{s}=\omega_{i}^{s}(\theta,\xi)\mathrm{d}\theta^{i}\,,
\]
and their exterior derivative can be written as
\begin{align}
\mathrm{d}\omega^{s} & =\partial_{\theta^{k}}\omega_{i}^{s}\mathrm{d}\theta^{k}\wedge\mathrm{d}\theta^{i}+\partial_{\xi^{j}}\omega_{i}^{s}\mathrm{d}\xi^{j}\wedge\mathrm{d}\theta^{i}\nonumber \\
 & =\partial_{\theta^{k}}\omega_{i}^{s}\mathrm{d}\theta^{k}\wedge\mathrm{d}\theta^{i}-L_{\partial_{\xi^{j}}}(\omega^{s})\wedge\mathrm{d}\xi^{j}\label{eq:dw_evaluated}
\end{align}
with $L_{\partial_{\xi^{j}}}(\omega^{s})$ denoting the Lie derivative
of $\omega^{s}$ along the vector field $\partial_{\xi^{j}}$. Because
of $L_{\partial_{\xi^{j}}}(\omega^{s})\in P_{k+1}^{+}=\mathrm{span}\{\omega,\rho\}$
(invariance of $P_{k+1}^{+}$ w.r.t. $\mathrm{span}\{\mathrm{d}f\}^{\perp}=\mathrm{span}\{\partial_{\xi}\}$)
we have
\begin{equation}
L_{\partial_{\xi^{j}}}(\omega^{s})\wedge\mathrm{d}\xi^{j}\wedge\omega^{1}\wedge\ldots\wedge\omega^{d_{1}}\wedge\rho^{1}\wedge\ldots\wedge\rho^{d_{3}}=0\,,\label{eq:Lie_derivative_wedge_product}
\end{equation}
and hence substituting (\ref{eq:dw_evaluated}) into (\ref{eq:dw_wedge_all})
yields
\begin{multline}
\partial_{\theta^{k}}\omega_{i}^{s}\mathrm{d}\theta^{k}\wedge\mathrm{d}\theta^{i}\wedge\omega^{1}\wedge\ldots\wedge\omega^{d_{1}}\\
\wedge\mu^{1}\wedge\ldots\wedge\mu^{d_{2}}\wedge\rho^{1}\wedge\ldots\wedge\rho^{d_{3}}=0\,,\label{eq:dw_wedge_all_evaluated}
\end{multline}
i.e., the terms with $L_{\partial_{\xi^{j}}}(\omega^{s})\wedge\mathrm{d}\xi^{j}$
vanish. Since there exists no linear combination of the 1-forms $\mu$
which is contained in $\mathrm{span}\{\mathrm{d}f\}=\mathrm{span}\{\mathrm{d}\theta\}$,
(\ref{eq:dw_wedge_all_evaluated}) is equivalent to
\begin{equation}
\partial_{\theta^{k}}\omega_{i}^{s}\mathrm{d}\theta^{k}\wedge\mathrm{d}\theta^{i}\wedge\omega^{1}\wedge\ldots\wedge\omega^{d_{1}}\wedge\rho^{1}\wedge\ldots\wedge\rho^{d_{3}}=0\,.\label{eq:dw_wedge_all_evaluated_reduced}
\end{equation}
Subtracting (\ref{eq:Lie_derivative_wedge_product}) from (\ref{eq:dw_wedge_all_evaluated_reduced})
and replacing $\partial_{\theta^{k}}\omega_{i}^{s}\mathrm{d}\theta^{k}\wedge\mathrm{d}\theta^{i}-L_{\partial_{\xi^{j}}}(\omega^{s})\wedge\mathrm{d}\xi^{j}$
by $\mathrm{d}\omega^{s}$ according to (\ref{eq:dw_evaluated}) finally
shows that (\ref{eq:WedgeProduct_1}) indeed holds. Next, we also
have to prove (\ref{eq:WedgeProduct_2}). We do this by induction,
using the fact that the 1-forms $\rho$ are just Lie derivatives (possibly
also higher order) of the 1-forms $\omega$ along vector fields $v\in\mathrm{span}\{\mathrm{d}f\}^{\perp}$.
Let us assume that a 1-form $\bar{\omega}\in P_{k+1}^{+}$ meets
\begin{equation}
\mathrm{d}\bar{\omega}\wedge\omega^{1}\wedge\ldots\wedge\omega^{d_{1}}\wedge\rho^{1}\wedge\ldots\wedge\rho^{d_{3}}=0\,,\label{eq:assumption_omega_bar}
\end{equation}
which we have just proven for the 1-forms $\omega^{1},\ldots,\omega^{d_{1}}$,
and consider the expression
\begin{equation}
\mathrm{d}\left(L_{v}\bar{\omega}\right)\wedge\omega^{1}\wedge\ldots\wedge\omega^{d_{1}}\wedge\rho^{1}\wedge\ldots\wedge\rho^{d_{3}}\,.\label{eq:expression_with_Lie_derivative}
\end{equation}
Because of $\mathrm{d}\left(L_{v}\bar{\omega}\right)=L_{v}\left(\mathrm{d}\bar{\omega}\right)$
and the property $L_{v}(\alpha\wedge\beta)=L_{v}(\alpha)\wedge\beta+\alpha\wedge L_{v}(\beta)$
of the Lie derivative for arbitrary (also higher-order) differential
forms $\alpha,\beta$, (\ref{eq:expression_with_Lie_derivative})
can be written as
\begin{equation}
-\mathrm{d}\bar{\omega}\wedge L_{v}\left(\omega^{1}\wedge\ldots\wedge\omega^{d_{1}}\wedge\rho^{1}\wedge\ldots\wedge\rho^{d_{3}}\right)\,,\label{eq:expression_with_Lie_derivative_after_product_rule}
\end{equation}
where we have already used the assumption (\ref{eq:assumption_omega_bar}).
Due to the invariance of $P_{k+1}^{+}=\mathrm{span}\{\omega,\rho\}$
w.r.t. vector fields $v\in\mathrm{span}\{\mathrm{d}f\}^{\perp}$,
(\ref{eq:expression_with_Lie_derivative_after_product_rule}) is of
the form
\begin{equation}
c\,\mathrm{d}\bar{\omega}\wedge\omega^{1}\wedge\ldots\wedge\omega^{d_{1}}\wedge\rho^{1}\wedge\ldots\wedge\rho^{d_{3}}\label{eq:expression_after_using_invariance}
\end{equation}
with some smooth function $c\in C^{\infty}(\mathcal{X}\times\mathcal{U})$.
However, because of the assumption (\ref{eq:assumption_omega_bar})
the expression (\ref{eq:expression_after_using_invariance}) vanishes,
and hence we have shown that (\ref{eq:assumption_omega_bar}) implies
\[
\mathrm{d}\left(L_{v}\bar{\omega}\right)\wedge\omega^{1}\wedge\ldots\wedge\omega^{d_{1}}\wedge\rho^{1}\wedge\ldots\wedge\rho^{d_{3}}=0\,.
\]
Since the 1-forms $\rho$ are Lie derivatives of the 1-forms $\omega$,
(\ref{eq:WedgeProduct_2}) follows and the proof is complete.
\end{proof}

\subsection{\protect\label{subsec:SystemDecompositionsForwardFlatness}System
Decompositions and Forward-Flatness}

As shown in \cite{KolarSchoberlDiwold:2019}, for forward-flat systems
(\ref{eq:sys}) there exists a sequence of repeated triangular decompositions
(\ref{eq:basic_decomposition_flat}) such that after the final step
the subsystem (\ref{eq:decomposition_subsystem}) is empty, i.e.,
$\dim(\bar{x}_{2})=0$. In the following, we will show that such a
sequence of decompositions exists if and only if the sequence of codistributions
(\ref{eq:sequence}) terminates with $P_{\bar{k}}=0$. In fact, straightening
out the integrable codistributions (\ref{eq:sequence}) by the Frobenius
theorem yields the state transformations (\ref{eq:decompostion_state_transformation})
which are required for the decompositions (\ref{eq:basic_decomposition_flat}).
First, we prove that for the original system (\ref{eq:sys}) a decomposition
(\ref{eq:basic_decomposition_flat}) according to Theorem \ref{thm:basic_decomposition_flat}
exists if and only if $\dim(P_{2})<\dim(P_{1})$. In this case, straightening
out $P_{2}$ by a state transformation and performing an additional
input transformation (\ref{eq:decompostion_input_transformation}),
which, as shown below, can be derived by a simple normalization of
the system equations, transforms the system (\ref{eq:sys}) into a
triangular form (\ref{eq:basic_decomposition_flat}).
\begin{thm}
\label{thm:relation_sequence_decomposition}A discrete-time system
(\ref{eq:sys}) can be transformed into a triangular form (\ref{eq:basic_decomposition_flat})
with $\dim(\bar{x}_{1})\geq1$ and $\mathrm{rank}(\partial_{\bar{u}_{1}}f_{1})=\dim(\bar{x}_{1})$
if and only if the corresponding sequence (\ref{eq:sequence}) meets
$\dim(P_{2})<\dim(P_{1})$.
\end{thm}

\begin{proof}
\ \\\emph{Necessity:} To prove the necessity of $\dim(P_{2})<\dim(P_{1})$,
we show that for a system of the form (\ref{eq:basic_decomposition_flat})
the condition $P_{2}\subset\mathrm{span}\{\mathrm{d}\bar{x}_{2}\}$
holds.\footnote{Note that $P_{2}$ corresponds to a decomposition (\ref{eq:basic_decomposition_flat})
with the minimal possible dimension of the subsystem (\ref{eq:decomposition_subsystem}).
There may also exist decompositions (\ref{eq:basic_decomposition_flat})
of a system (\ref{eq:sys}) with $\dim(\bar{x}_{2})>\dim(P_{2})$.} Because of $\mathrm{rank}(\partial_{\bar{u}_{1}}f_{1})=\dim(\bar{x}_{1})$
and the fact that the functions $f_{2}$ are independent of $\bar{u}_{1}$,
all linear combinations of the differentials $\mathrm{d}f$ which
are contained in $P_{1}=\mathrm{span}\{\mathrm{d}\bar{x}\}$ are linear
combinations of the differentials $\mathrm{d}f_{2}$ alone. Thus,
a system of the form (\ref{eq:basic_decomposition_flat}) meets
\[
P_{1}\cap\mathrm{span}\{\mathrm{d}f\}=P_{1}\cap\mathrm{span}\{\mathrm{d}f_{2}\}\subset\mathrm{span}\{\mathrm{d}f_{2}\}\,.
\]
Since $\mathrm{span}\{\mathrm{d}f_{2}\}$ contains $P_{1}\cap\mathrm{span}\{\mathrm{d}f\}$
and is obviously invariant w.r.t. $\mathrm{span}\{\mathrm{d}f\}^{\perp}$,
it must also contain $P_{2}^{+}$, which is by definition the smallest
codistribution with these properties. A backward-shift of the relation
$P_{2}^{+}\subset\mathrm{span}\{\mathrm{d}f_{2}\}$ then yields $P_{2}\subset\mathrm{span}\{\mathrm{d}\bar{x}_{2}\}$,
and because of $P_{1}=\mathrm{span}\{\mathrm{d}\bar{x}_{1},\mathrm{d}\bar{x}_{2}\}$
with $\dim(\bar{x}_{1})\geq1$ we immediately get $\dim(P_{2})<\dim(P_{1})$.\\
\emph{Sufficiency:} To prove the sufficiency of $\dim(P_{2})<\dim(P_{1})$,
we show how the coordinate transformation (\ref{eq:decomposition_coord_transformation})
that achieves the decomposition (\ref{eq:basic_decomposition_flat})
can be derived. First, since $P_{2}$ is integrable and $P_{2}\subset\mathrm{span}\{\mathrm{d}x\}$,
there exists a state transformation (\ref{eq:decompostion_state_transformation})
with
\[
\dim(\bar{x}_{1})=\dim(P_{1})-\dim(P_{2})\geq1
\]
such that $P_{2}=\mathrm{span}\{\mathrm{d}\bar{x}_{2}\}$. For the
resulting system\begin{subequations}\label{eq:P2_straightened_out_sys}
\begin{align}
\bar{x}_{2}^{+} & =f_{2}(\bar{x}_{2},\bar{x}_{1},u)\label{eq:P2_straightened_out_subsys}\\
\bar{x}_{1}^{+} & =f_{1}(\bar{x}_{2},\bar{x}_{1},u)\,,\label{eq:P2_straightened_out_feedback}
\end{align}
\end{subequations}we perform an input transformation (\ref{eq:decompostion_input_transformation})
with $\dim(\bar{u}_{2})=\mathrm{rank}(\partial_{u}f_{2})$ such that
$\dim(\bar{u}_{2})$ equations of the subsystem (\ref{eq:P2_straightened_out_subsys})
are simplified to $\bar{x}_{2}^{i_{2},+}=\bar{u}_{2}^{i_{2}}$ (by
just setting $\bar{u}_{2}^{i_{2}}=f_{2}^{i_{2}}(\bar{x}_{2},\bar{x}_{1},u)$).
After this input transformation the system (\ref{eq:P2_straightened_out_sys})
must have the form (\ref{eq:basic_decomposition_flat}) with $f_{2}$
independent of $\bar{u}_{1}$, since otherwise $\mathrm{rank}(\partial_{\bar{u}}f_{2})>\dim(\bar{u}_{2})$.
Now let us show that also the condition $\mathrm{rank}(\partial_{\bar{u}_{1}}f_{1})=\dim(\bar{x}_{1})$
indeed holds. In the case $\mathrm{rank}(\partial_{\bar{u}_{1}}f_{1})<\dim(\bar{x}_{1})$
there would exist a nonzero linear combination of the differentials
$\mathrm{d}f_{1}$ which is contained in $\mathrm{span}\{\mathrm{d}\bar{x},\mathrm{d}\bar{u}_{2}\}$,
and because of $\mathrm{rank}(\partial_{\bar{u}_{2}}f_{2})=\dim(\bar{u}_{2})$
the differentials $\mathrm{d}\bar{u}_{2}$ of this linear combination
could be cancelled out by performing a further linear combination
with the differentials $\mathrm{d}f_{2}$. In other words, there would
exist a linear combination $\omega$ of the differentials $\mathrm{d}f_{1}$
and $\mathrm{d}f_{2}$ which is contained in $P_{1}=\mathrm{span}\{\mathrm{d}\bar{x}\}$
and involves at least one of the differentials $\mathrm{d}f_{1}$.
That is, $\omega\in P_{1}\cap\mathrm{span}\{\mathrm{d}f\}$ but $\omega\notin\mathrm{span}\{\mathrm{d}f_{2}\}$.\footnote{Because of the submersivity property (\ref{eq:submersivity}), a linear
combination which involves at least one of the differentials $\mathrm{d}f_{1}$
cannot be contained in $\mathrm{span}\{\mathrm{d}f_{2}\}$.} However, according to Step 2 of Algorithm \ref{alg:definition_sequence},
with $P_{2}^{+}=\mathrm{span}\{\mathrm{d}f_{2}\}$ due to the above
state transformation we have $P_{1}\cap\mathrm{span}\{\mathrm{d}f\}\subset\mathrm{span}\{\mathrm{d}f_{2}\}$,
which is a contradiction.
\end{proof}
With this theorem, we have established a connection between a first
decomposition step (\ref{eq:basic_decomposition_flat}) for a system
(\ref{eq:sys}) and the sequence (\ref{eq:sequence}). However, for
checking the forward-flatness of a system (\ref{eq:sys}), in general
several decomposition steps are needed. In order to establish a connection
between the sequence (\ref{eq:sequence}) and repeated decompositions
of the form (\ref{eq:basic_decomposition_flat}), we prove the following.
\begin{lem}
\label{lem:sequence_subsystem}Consider a system (\ref{eq:sys}) with
the corresponding sequence (\ref{eq:sequence}) as well as a decomposition
(\ref{eq:basic_decomposition_flat}) such that $P_{2}=\mathrm{span}\{\mathrm{d}\bar{x}_{2}\}$.
If $P_{k}^{\prime}$, $k\geq1$ denotes the sequence (\ref{eq:sequence})
computed for the subsystem (\ref{eq:decomposition_subsystem}) on
a smaller-dimensional manifold with coordinates $(\bar{x}_{2},\bar{x}_{1},\bar{u}_{2})$,
then $P_{k}^{\prime}=P_{k+1}$, $k\geq1$.
\end{lem}

\begin{proof}
For $k=1$, because of $P_{1}^{\prime}=\mathrm{span}\{\mathrm{d}\bar{x}_{2}\}$
it is obvious that $P_{1}^{\prime}=P_{2}$. For $k>1$, the proof
can be reduced to the question whether a computation of the sequence
according to Algorithm \ref{alg:definition_sequence} on a smaller-dimensional
manifold $\mathcal{X}\times\mathcal{U}_{2}$ with coordinates $(\bar{x}_{2},\bar{x}_{1},\bar{u}_{2})$
and $\mathrm{d}f_{2}$ instead of $\mathrm{d}f$ yields the same result
as a computation on the original manifold $\mathcal{X}\times\mathcal{U}$
with the additional coordinates $\bar{u}_{1}$ and $\mathrm{d}f$.
In Step 1 of Algorithm \ref{alg:definition_sequence}, because of
$\mathrm{rank}(\partial_{\bar{u}_{1}}f_{1})=\dim(\bar{x}_{1})$ and
the fact that the functions $f_{2}$ are independent of $\bar{u}_{1}$,
it does not make a difference whether a codistribution $P_{k}\subset\mathrm{\mathrm{span}}\{\mathrm{d}\bar{x}\}$
is intersected with $\mathrm{\mathrm{span}}\{\mathrm{d}f\}$ or $\mathrm{\mathrm{span}}\{\mathrm{d}f_{2}\}$.
Regarding Step 2, assume there has been performed an additional input
transformation $(\hat{u}_{1},\tilde{u}_{1})=\Phi_{u}(\bar{x},\bar{u})$
which replaces $\bar{u}_{1}$ such that the equations (\ref{eq:decomposition_feedback})
are simplified to the form $\bar{x}_{1}^{i_{1},+}=\hat{u}_{1}^{i_{1}}$,
$i_{1}=1,\ldots,\dim(\bar{x}_{1})$. Because of $\mathrm{rank}(\partial_{\bar{u}_{1}}f_{1})=\dim(\bar{x}_{1})$
this is always possible. Then it can be observed that in contrast
to the distribution $\mathrm{\mathrm{span}}\{\mathrm{d}f_{2}\}^{\perp}$
on $\mathcal{X}\times\mathcal{U}_{2}$, the distribution $\mathrm{\mathrm{span}}\{\mathrm{d}f\}^{\perp}$
on $\mathcal{X}\times\mathcal{U}$ is larger since it contains the
additional vector fields $\partial_{\tilde{u}_{1}}$ corresponding
to the redundant inputs $\tilde{u}_{1}$ (provided that the system
(\ref{eq:sys}) has redundant inputs, $\dim(\tilde{u}_{1})=m-\mathrm{rank}(\partial_{u}f)$).
However, since the considered codistributions are invariant w.r.t.
the vector fields $\partial_{\tilde{u}_{1}}$ anyway, also Step 2
yields the same result in both cases. Finally, since Step 3 consists
only in a backward-shift, indeed $P_{k}^{\prime}=P_{k+1}$, $k\geq1$.
\end{proof}
With Lemma \ref{lem:sequence_subsystem}, we can now prove our main
result.
\begin{thm}
\label{thm:test_forward_flatness}A system (\ref{eq:sys}) is forward-flat
if and only if the sequence (\ref{eq:sequence}) terminates with $P_{\bar{k}}=0$.
\end{thm}

\begin{proof}
From Theorem \ref{thm:basic_decomposition_flat} and Lemma \ref{lem:flatness_properties_basic_decomposition},
it is clear that a system (\ref{eq:sys}) is forward-flat if and only
if it can be decomposed repeatedly until in some step the subsystem
(\ref{eq:decomposition_subsystem}) is trivial with $\dim(\bar{x}_{2})=0$.
Because of Theorem \ref{thm:relation_sequence_decomposition} and
Lemma \ref{lem:sequence_subsystem}, the existence of these repeated
decompositions can be checked by computing the codistributions of
sequence (\ref{eq:sequence}). They correspond to the codistributions
spanned by the differentials of the state variables of the successively
computed subsystems (\ref{eq:decomposition_subsystem}).\footnote{Assuming that in every step the decomposition is performed such that
$\dim(\bar{x}_{2})$ is minimal.} Thus, the last subsystem (\ref{eq:decomposition_subsystem}) is trivial
if and only if $P_{\bar{k}}=0$.
\end{proof}
If Theorem \ref{thm:test_forward_flatness} confirms the forward-flatness
of a system (\ref{eq:sys}), a flat output can be obtained systematically
by performing the repeated decompositions (\ref{eq:basic_decomposition_flat})
which are induced by the sequence (\ref{eq:sequence}) and applying
Lemma \ref{lem:flatness_properties_basic_decomposition}. These decompositions
can be derived by successively straightening out the integrable codistributions
$P_{k}$, $k\geq2$ by state transformations (\ref{eq:decompostion_state_transformation})
and transforming the resulting systems (\ref{eq:P2_straightened_out_sys})
into the form (\ref{eq:basic_decomposition_flat}) by further input
transformations (\ref{eq:decompostion_input_transformation}), as
it is shown in the sufficiency part of the proof of Theorem \ref{thm:relation_sequence_decomposition}
for the first decomposition step with $k=2$. A flat output of the
last, trivial subsystem (\ref{eq:decomposition_subsystem}) with $\dim(\bar{x}_{2})=0$
corresponding to $P_{\bar{k}}=0$ is given by its inputs $(\bar{x}_{1},\bar{u}_{2})$,
and by adding the redundant inputs $\tilde{u}_{1}$ according to item
\ref{enu:case_redundant_inputs_basic_decomposition} of Lemma \ref{lem:flatness_properties_basic_decomposition}
for all decomposition steps, a flat output $y$ of the original system
(\ref{eq:sys}) can be obtained. To get the flat output in original
coordinates $(x,u)$, it is of course necessary to apply the corresponding
inverse coordinate transformations.
\begin{rem}
Systems which are linearizable by static feedback are contained in
the class of forward-flat systems, and hence the sequence (\ref{eq:sequence})
also terminates with $P_{\bar{k}}=0$. However, Step 2 of Algorithm
\ref{alg:definition_sequence} is always trivial since $P_{k}\cap\mathrm{span}\{\mathrm{d}f\}$
is already invariant and hence $P_{k+1}^{+}=P_{k}\cap\mathrm{span}\{\mathrm{d}f\}$
for all $k\geq1$, cf. also Remark \ref{rem:static_feedback_decomposition}.
\end{rem}

\section{\protect\label{sec:Example}Examples}

In this section, we illustrate our results by two examples.

\subsection{Academic Example}

Consider the system
\begin{equation}
\begin{aligned}x^{1,+} & =x^{2}(u^{1}+1)\,, & x^{4,+} & =x^{5}+1-\tfrac{x^{1}(u^{1}+1)}{x^{2}+1}\\
x^{2,+} & =u^{1}\,, & x^{5,+} & =x^{2}+u^{2}\\
x^{3,+} & =x^{4}+u^{2}-1\,,
\end{aligned}
\label{eq:academic_ex_sys}
\end{equation}
with the equilibrium $x_{0}=(0,0,0,1,0)$ and $u_{0}=(0,0)$, which
is not static feedback linearizable and has also been studied in \cite{KolarSchoberlSchlacher:2016-2}
in the context of implicit system decompositions. To check its forward-flatness,
we compute the sequence (\ref{eq:sequence}) according to Algorithm
\ref{alg:definition_sequence}. The first codistribution of (\ref{eq:sequence})
is given by
\[
P_{1}=\mathrm{span}\{\mathrm{d}x\}=\mathrm{span}\{\mathrm{d}x^{1},\mathrm{d}x^{2},\mathrm{d}x^{3},\mathrm{d}x^{4},\mathrm{d}x^{5}\}\,.
\]
As explained before Algorithm \ref{alg:definition_sequence}, it is
convenient to perform the computations in adapted coordinates
\begin{equation}
\begin{array}{ccl}
\theta^{i} & = & f^{i}(x,u)\,,\quad i=1,\ldots,n\\
\xi^{1} & = & x^{1}\\
\xi^{2} & = & x^{3}\,,
\end{array}\label{eq:academic_ex_adapted_coordinates}
\end{equation}
where the functions $h(x,u)$ of (\ref{eq:adapted_coordinates}) have
been chosen such that the transformation is invertible. For $k=1$,
the intersection of Step 1 yields
\begin{multline*}
P_{1}\cap\mathrm{span}\{\mathrm{d}\theta\}=\mathrm{span}\{(\theta^{2}+1)\mathrm{d}\theta^{1}-\theta^{1}\mathrm{d}\theta^{2},\mathrm{d}\theta^{3}-\mathrm{d}\theta^{5},\\
\xi(\theta^{2}+1)\mathrm{d}\theta^{2}+\left(\theta^{1}+\theta^{2}+1\right)\mathrm{d}\theta^{4}\}\,,
\end{multline*}
which is not yet invariant w.r.t. the distribution $\mathrm{span}\{\mathrm{d}\theta\}^{\perp}=\mathrm{span}\{\partial_{\xi}\}$.
However, adding in Step 2 the 1-form $(\theta^{2}+1)\mathrm{d}\theta^{2}$,
which is the Lie derivative of the last 1-form of the above basis
of $P_{1}\cap\mathrm{span}\{\mathrm{d}\theta\}$ w.r.t. the vector
field $\partial_{\xi^{1}}$, results in an invariant codistribution
$P_{2}^{+}$. A basis for $P_{2}^{+}$ which is independent of the
coordinates $\xi$ is given by\textbf{\emph{
\[
P_{2}^{+}=\mathrm{span}\{\mathrm{d}\theta^{1},\mathrm{d}\theta^{2},\mathrm{d}\theta^{3}-\mathrm{d}\theta^{5},\mathrm{d}\theta^{4}\}\,,
\]
}}and because of $\theta^{i}=f^{i}(x,u)$, the backward-shift of Step
3 yields $P_{2}$ in original coordinates as
\[
P_{2}=\mathrm{span}\{\mathrm{d}x^{1},\mathrm{d}x^{2},\mathrm{d}x^{3}-\mathrm{d}x^{5},\mathrm{d}x^{4}\}\,.
\]
Continuing Algorithm \ref{alg:definition_sequence} with $k=2$ and
using again the adapted coordinates (\ref{eq:academic_ex_adapted_coordinates}),
Step 1 yields
\[
P_{2}\cap\mathrm{span}\{\mathrm{d}\theta\}=\mathrm{span}\{(\theta^{2}+1)\mathrm{d}\theta^{1}-\theta^{1}\mathrm{d}\theta^{2},\mathrm{d}\theta^{3}-\mathrm{d}\theta^{5}\}\,,
\]
which is already invariant w.r.t. $\mathrm{span}\{\mathrm{d}\theta\}^{\perp}=\mathrm{span}\{\partial_{\xi}\}$.
Thus, Step 2 is trivial with $P_{3}^{+}=P_{2}\cap\mathrm{span}\{\mathrm{d}\theta\}$,
and after the backward-shift of Step 3 we get
\[
P_{3}=\mathrm{span}\{(x^{2}+1)\mathrm{d}x^{1}-x^{1}\mathrm{d}x^{2},\mathrm{d}x^{3}-\mathrm{d}x^{5}\}
\]
in original coordinates. Finally, for $k=3$, the intersection of
Step 1 yields $P_{3}\cap\mathrm{span}\{\mathrm{d}\theta\}=0$. Hence,
$P_{4}=0$ and according to Theorem \ref{thm:test_forward_flatness}
the system (\ref{eq:academic_ex_sys}) is forward-flat. As discussed
at the end of Section \ref{subsec:SystemDecompositionsForwardFlatness},
a flat output can be obtained systematically by performing repeated
system decompositions (\ref{eq:basic_decomposition_flat}) which are
induced by the sequence $P_{4}\subset P_{3}\subset P_{2}\subset P_{1}$.
For this purpose, the codistributions of the sequence $P_{4}\subset P_{3}\subset P_{2}\subset P_{1}$
are straightened out based on the Frobenius theorem with the state
transformation
\begin{equation}
\begin{array}{ccl}
\bar{x}_{3}^{1} & = & \tfrac{x^{1}}{x^{2}+1}\\
\bar{x}_{3}^{2} & = & x^{3}-x^{5}\\
\bar{x}_{2}^{1} & = & x^{4}\\
\bar{x}_{2}^{2} & = & x^{2}\\
\bar{x}_{1}^{1} & = & x^{5}\,.
\end{array}\label{eq:academic_ex_state_transformation}
\end{equation}
In these coordinates, the codistributions are given by
\[
\begin{aligned}P_{1} & =\mathrm{span}\{\mathrm{d}\bar{x}_{3}^{1},\mathrm{d}\bar{x}_{3}^{2},\mathrm{d}\bar{x}_{2}^{1},\mathrm{d}\bar{x}_{2}^{2},\mathrm{d}\bar{x}_{1}^{1}\}\\
P_{2} & =\mathrm{span}\{\mathrm{d}\bar{x}_{3}^{1},\mathrm{d}\bar{x}_{3}^{2},\mathrm{d}\bar{x}_{2}^{1},\mathrm{d}\bar{x}_{2}^{2}\}\\
P_{3} & =\mathrm{span}\{\mathrm{d}\bar{x}_{3}^{1},\mathrm{d}\bar{x}_{3}^{2}\}\\
P_{4} & =0\,,
\end{aligned}
\]
and the transformed system reads
\begin{equation}
\begin{aligned}\bar{x}_{3}^{1,+} & =\bar{x}_{2}^{2}\\
\bar{x}_{3}^{2,+} & =\bar{x}_{2}^{1}-1-\bar{x}_{2}^{2}\\
\bar{x}_{2}^{1,+} & =\bar{x}_{1}^{1}+1-\bar{x}_{3}^{1}(u^{1}+1)\\
\bar{x}_{2}^{2,+} & =u^{1}\\
\bar{x}_{1}^{1,+} & =\bar{x}_{2}^{2}+u^{2}\,.
\end{aligned}
\label{eq:academic_ex_transformed_sys}
\end{equation}
In fact, we can see here already a triangular structure corresponding
to repeated decompositions of the form (\ref{eq:basic_decomposition_flat}).
The subsystem (\ref{eq:decomposition_subsystem}) of the first decomposition
step has state variables $(\bar{x}_{3}^{1},\bar{x}_{3}^{2},\bar{x}_{2}^{1},\bar{x}_{2}^{2})$
and inputs $(\bar{x}_{1}^{1},u^{1})$, the subsystem (\ref{eq:decomposition_subsystem})
of the second decomposition step has state variables $(\bar{x}_{3}^{1},\bar{x}_{3}^{2})$
and inputs $(\bar{x}_{2}^{1},\bar{x}_{2}^{2})$, and after the third
decomposition step the remaining subsystem (\ref{eq:decomposition_subsystem})
is trivial with no state and inputs $(\bar{x}_{3}^{1},\bar{x}_{3}^{2})$.\footnote{It should be noted that achieving decompositions of the form (\ref{eq:basic_decomposition_flat})
requires in general not only state- but also input transformations,
cf. Theorem \ref{thm:basic_decomposition_flat}. In the sufficiency
part of the proof of Theorem \ref{thm:relation_sequence_decomposition},
it is shown how to derive such input transformations in a straightforward
way. Furthermore, if e.g. in the second decomposition step an input
transformation for the subsystem (\ref{eq:decomposition_subsystem})
is also applied to the equations (\ref{eq:decomposition_feedback})
of the complete system (which is not necessary at all for computing
flat outputs), it does in general not preserve the state representation
of the latter (since besides original input variables also the state
variables $\bar{x}_{1}$ of the complete system (\ref{eq:basic_decomposition_flat})
serve as input variables for the subsystem (\ref{eq:decomposition_subsystem})).
Hence, combining the coordinate transformations of all decomposition
steps and applying the resulting transformation to the original system
(\ref{eq:sys}) yields in general a structurally flat implicit triangular
system representation, and not necessarily an explicit one like (\ref{eq:academic_ex_transformed_sys}).
However, it still would allow to read off a flat output and systematically
determine the parameterization (\ref{eq:flat_param}). } Based on this triangular structure, it can be verified that $y=(\bar{x}_{3}^{1},\bar{x}_{3}^{2})$
is a flat output. The first and the second equation of (\ref{eq:academic_ex_transformed_sys})
can be used to calculate $\bar{x}_{2}^{1}$ and $\bar{x}_{2}^{2}$
as a function of $y$ and its forward-shifts. In a second step, the
parameterization of $\bar{x}_{1}^{1}$ and $u^{1}$ can be calculated
from the equations three and four of (\ref{eq:academic_ex_transformed_sys}).
Finally, from the last equation of (\ref{eq:academic_ex_transformed_sys}),
we also get $u^{2}$ as a function of $y$ and its forward-shifts,
which completes the map (\ref{eq:flat_param}). By applying the inverse
transformation of (\ref{eq:academic_ex_state_transformation}), the
flat output in original coordinates can be obtained as $y=(\frac{x^{1}}{x^{2}+1},x^{3}-x^{5})$.
However, it is important to emphasize again that for only checking
whether the system is forward-flat or not, in contrast to the approach
proposed in \cite{KolarSchoberlDiwold:2019}, the calculation of a
flat output is not required. Computing the sequence (\ref{eq:sequence})
and applying Theorem \ref{thm:test_forward_flatness} is sufficient.

\subsection{VTOL Aircraft}

As a second example, let us consider the planar VTOL aircraft discussed
e.g. in \cite{Sastry:1999} or \cite{FliessLevineMartinRouchon:1999},
which is described by the continuous-time dynamics
\begin{equation}
\begin{array}{ll}
\dot{x}=v_{x}\,,\quad & \dot{v}_{x}=\varepsilon\cos(\theta)u^{2}-\sin(\theta)u^{1}\\
\dot{z}=v_{z}\,, & \dot{v}_{z}=\cos(\theta)u^{1}+\varepsilon\sin(\theta)u^{2}-1\\
\dot{\theta}=\omega\,, & \dot{\omega}=u^{2}\,.
\end{array}\label{eq:VTOL_cont}
\end{equation}
It is well-known that this system is flat with a flat output $y=(x-\varepsilon\sin(\theta),z+\varepsilon\cos(\theta))$.
By combining an Euler discretization for some sampling time $T_{s}>0$
with a suitable prior state transformation which transforms (\ref{eq:VTOL_cont})
into a structurally flat triangular form as it is shown in \cite{DiwoldKolarSchoberl:2022}
for a gantry crane, a flat discrete-time system
\begin{equation}
\begin{aligned}x^{1,+} & =x^{1}+T_{s}x^{3}\\
x^{2,+} & =x^{2}+T_{s}x^{4}\\
x^{3,+} & =x^{3}+T_{s}\sin(x^{5})(\varepsilon(x^{6})^{2}-u^{1})\\
x^{4,+} & =x^{4}+T_{s}\cos(x^{5})(-\varepsilon(x^{6})^{2}+u^{1})-T_{s}\\
x^{5,+} & =x^{5}+T_{s}x^{6}\\
x^{6,+} & =x^{6}+T_{s}u^{2}
\end{aligned}
\label{eq:VTOL_disc}
\end{equation}
with a flat output $y=(x^{1},x^{2})$ can be obtained. Computing the
sequence (\ref{eq:sequence}) for this system yields
\begin{equation}
\begin{aligned}P_{1} & =\mathrm{span}\{\mathrm{d}x^{1},\mathrm{d}x^{2},\mathrm{d}x^{3},\mathrm{d}x^{4},\mathrm{d}x^{5},\mathrm{d}x^{6}\}\\
P_{2} & =\mathrm{span}\{\mathrm{d}x^{1},\mathrm{d}x^{2},\mathrm{d}x^{3},\mathrm{d}x^{4},\mathrm{d}x^{5}\}\\
P_{3} & =\mathrm{span}\{\mathrm{d}x^{1},\mathrm{d}x^{2},\mathrm{d}x^{3},\mathrm{d}x^{4}\}\\
P_{4} & =\mathrm{span}\{\mathrm{d}x^{1},\mathrm{d}x^{2}\}\\
P_{5} & =0\,,
\end{aligned}
\label{eq:VTOL_sequence}
\end{equation}
which confirms its forward-flatness. The fact that the codistributions
(\ref{eq:VTOL_sequence}) are already straightened out is due to the
structurally flat triangular form of (\ref{eq:VTOL_disc}). For $k=1$
and $k=2$, Step 2 of Algorithm \ref{alg:definition_sequence} is
nontrivial, since $P_{1}\cap\mathrm{span}\{\mathrm{d}f\}$ and $P_{2}\cap\mathrm{span}\{\mathrm{d}f\}$
are not invariant w.r.t. $\mathrm{span}\{\mathrm{d}f\}^{\perp}$ and
a Lie derivative has to be added. Thus, the system is not static feedback
linearizable.

\section{\protect\label{sec:Conclusion}Conclusion}

We have derived a test for forward-flatness based on a unique sequence
of integrable codistributions (\ref{eq:sequence}), which can be considered
as dual version of the test proposed in \cite{KolarDiwoldSchoberl:2019}.
The sequence of integrable codistributions gives rise to repeated
system decompositions of the form (\ref{eq:basic_decomposition_flat}),
where the complete system is forward-flat if and only if the subsystem
(\ref{eq:decomposition_subsystem}) is forward-flat. The dimension
of the last codistribution of the sequence corresponds to the dimension
of the last subsystem, for which no further decomposition exists.
Since the existence of such decompositions is a necessary condition
for forward-flatness, the original system (\ref{eq:sys}) is forward-flat
if and only if the last codistribution has dimension zero. The only
additional effort compared to a test for static feedback linearizability
consists in the computation of the smallest invariant codistributions
in Step 2 of Algorithm \ref{alg:definition_sequence}, which can be
achieved in a straightforward way by adding suitable Lie derivatives
of 1-forms. For static feedback linearizable systems this is trivial,
since already the codistributions of Step 1 are invariant. The computational
effort is also lower than in \cite{KolarDiwoldSchoberl:2019}, where
compared to a static feedback linearization test an additional calculation
of largest projectable subdistributions is required. The computation
of the smallest invariant codistributions in the presented dual approach
by just adding Lie derivatives of 1-forms can be considered here as
the simpler task. Furthermore, the calculations in \cite{KolarDiwoldSchoberl:2019}
involve two manifolds, whereas here all calculations are performed
on only one manifold. Future research will address extensions of the
presented results to flatness in the more general sense of \cite{DiwoldKolarSchoberl:2020},
pursuing similar ideas as in \cite{Kaldmae:2021} while trying to
keep the computational effort as low as possible.

\begin{ack}                               
This research was funded in whole, or in part, by the Austrian Science Fund (FWF) P36473. For the purpose of open access, the author has applied a CC BY public copyright licence to any Author Accepted Manuscript version arising from this submission.
\end{ack}

\bibliographystyle{plain}        
\bibliography{Bibliography}



\end{document}